\documentclass[11pt, reqno]{amsart}
\usepackage{amsfonts}
\usepackage{bbm}
\usepackage{amscd,amsfonts}
\usepackage{amssymb, eucal, amsfonts, amsmath, xypic, latexsym}
\usepackage{pifont}
\usepackage{mathrsfs,color}
\usepackage{amsthm,indentfirst,bm,fancyhdr,dsfont}
\usepackage{graphicx}
\usepackage[all]{xy}
\usepackage[CJKbookmarks=true]{hyperref}

\usepackage{mathrsfs}
\usepackage{amsmath}
\usepackage{amssymb}
\usepackage{hyperref}

\newtheorem{theorem}{Theorem}[section]
\newtheorem{lemma}[theorem]{Lemma}
\newtheorem{prop}[theorem]{Proposition}

\theoremstyle{definition}

\newtheorem{defn}[theorem]{Definition}

\newtheorem{rem}[theorem]{Remark}
\newtheorem{remark}[theorem]{Remark}

\newtheorem{question}[theorem]{Question}

\numberwithin{equation}{section}

\def\ggg{\mathfrak{g}}

\def\ca{\mathcal{A}}

\def\ct{\mathcal{T}}

\def\cz{\mathcal{Z}}
\def\cq{\mathcal{Q}}
\def\HC{\mathcal{HC}}
\def\cs{\mathcal{S}}

\def\sR{\mathsf{R}}
\def\sr{\mathsf{r}}
\def\sj{\mathsf{j}}

\def\kkk{\mathfrak{k}}
\def\ggg{\mathfrak{g}}
\def\mmm{\mathfrak{m}}
\def\ppp{\mathfrak{p}}

\def\hhh{\mathfrak{h}}

\def\bbb{\mathfrak{b}}

\def\nnn{\mathfrak{n}}

\def\bba{\mathbb{A}}
\def\bbc{\mathbb{C}}
\def\bbf{\mathbb{F}}
\def\bbz{\mathbb{Z}}

\def\bbk{{\mathds{k}}}

\def\bg{{\mathbf{g}}}
\def\bt{{\mathbf{t}}}
\def\df{{\mathbf{f}}}

\def\sfd{\textsf{d}}

\def\Lie{\text{Lie}}
\def\ad{\text{ad}}
\def\ker{\text{Ker}}

\def\Ad{\text{Ad}}
\def\End{{\text{End}}}
\def\Ext{{\text{Ext}}}

\def\Gr{\text{gr}}

\def\Frac{\mbox{Frac}}
\def\rank{\text{rank}}
\def\op{\text{op}}

\def\Specm{\text{Specm}}
\def\reg{\text{reg}}

\def\sing{\text{sing}}

{\vskip-\lastskip\medskip
  \noindent
  {\em #1.}\enspace
  }%
{\qed\par\medskip
  }

\begin{document}

\title[Centers and Azumaya loci of finite $W$-algebras]{Centers and Azumaya loci for finite $W$-algebras in positive characteristic}
\author{Bin Shu and Yang Zeng}

\address{School of Mathematical Sciences, East China Normal University, Shanghai 200241, China}
\email{bshu@math.ecnu.edu.cn}

\address{School of Statistics and Mathematics, Nanjing Audit University, Nanjing, Jiangsu Province 211815, China}
\email{zengyang214@163.com}
\thanks{
{{\it{Mathematics Subject Classification}} (2010): Primary 17B50, Secondary 17B05 and 17B08.
This work is  supported partially by the NSFC (Nos. 11671138; 11701284; 11771279), Shanghai Key Laboratory of PMMP (No. 13dz2260400).
}}

\begin{abstract}
In this paper, we study  the center $Z$ of the finite $W$-algebra $\ct(\ggg,e)$ associated with a semi-simple Lie algebra $\ggg$ over an algebraically closed field $\bbk$ of characteristic $p\gg0$, and an arbitrarily given nilpotent element $e\in\ggg$. We obtain an analogue of Veldkamp's theorem on the center. 
For the maximal spectrum $\Specm(Z)$, we show that its Azumaya locus  coincides with its smooth locus of smooth points. The former locus  reflects irreducible representations of  maximal dimension for $\ct(\ggg,e)$.
\end{abstract}
\maketitle
\setcounter{tocdepth}{1}
\tableofcontents
\section*{Introduction}
\subsection{}
A finite $W$-algebra $U(\ggg_\bbc,\hat e)$ is a certain associative algebra associated to a complex semi-simple Lie algebra ${\ggg_\bbc}$ and a nilpotent  element $\hat e\in{\ggg}_\bbc$. The study of finite $W$-algebras can be traced back to Kostant's work in the case when $\hat e$ is regular \cite{Ko}, whose construction was generalized to arbitrary even nilpotent elements by Lynch \cite{Ly}.  Premet developed the finite $W$-algebras in full generality in \cite{Pre1}. After the proof of the celebrated Kac-Weisfeiler conjecture for Lie algebras of reductive groups in \cite{P1}, Premet first constructed a modular version of finite $W$-algebras $U_\chi(\ggg,e)$ in \cite{Pre1} (they will be called the reduced extended $W$-algebras in this paper). By means of a complicated but natural ``admissible'' procedure, the finite $W$-algebras over $\bbc$ arise from the modular version,  and they are shown to be  filtrated deformations of the coordinate rings of Slodowy slices. The most important ingredient there is the construction of the PBW basis of finite $W$-algebras (cf. \cite[\S4]{Pre1}).

As a counterpart of the finite $W$-algebra $U(\ggg_\bbc,\hat e)$,  Premet introduced in \cite{Pre3} the finite $W$-algebra $\ct(\ggg,e)$ (which is denoted by $U(\ggg,e)$ there) 
over ${\bbk}=\overline{\mathbb{F}}_p$, and also the extended finite $W$-algebra $U(\ggg,e)$ (denoted by $\widehat U(\ggg,e)$ in \cite{Pre3}). In the same paper, the $p$-centers of $U(\ggg,e)$ are introduced. 
In his further work \cite{Pre2}, the $p$-center of the finite $W$-algebra $\ct(\ggg,e)$ is also introduced, and the relation between $\ct(\ggg,e)$ and the reduced $W$-algebra $\ct_\eta(\ggg,e)$ with $\eta\in\chi+\mmm^\perp$ (denoted by $U_\eta(\ggg,e)$ there) is discussed.

In the work of Premet, the ${\bbk}$-algebra $\ct(\ggg,e)$ is obtained from the $\bbc$-algebra $U(\ggg_\bbc,\hat e)$ through ``reduction modulo $p$". Hence the characteristic of field ${\bbk}$ must meet the condition that $p\in\Pi(A)$ for some admissible ring $A$ there (see \S\ref{211} for more details), thus sufficiently large. In our arguments, some will follow Premet's result. So we will assume that the characteristic $p$ of the base field $\bbk$ will be big enough in the present paper.
In a recent work of Goodwin-Topley \cite{GT}, the authors generalized Premet's work to the case with $p=\text{char}(\bbk)$ satisfying Jantzen's standard hypotheses (dependent on  the corresponding reductive algebraic group $G$ of $\ggg$). Roughly speaking, Goodwin-Topley's arguments are still valid to the most statements of the present paper. For the simplicity, we will do not judge which statements are under less restriction on $p$.

\subsection{}
The main purpose of the present paper is to develop  the theory of centers and the associated Azumaya property for the finite $W$-algebra $\ct(\ggg,e)$ in positive characteristic $p\gg0$. Our approach is to  generalize the classical theory of the centers of the universal enveloping algebra $U(\ggg)$ of a Lie algebra $\ggg$, which was first developed  by Veldkamp (see \cite{Ve}, \cite{KW} and \cite{MiRu}); and also applying Brown-Goodearl's arguments in \cite{BGl} on the Azumaya property for the algebras satisfying Auslander-regular and Maculay conditions to the modular finite $W$-algebra case.
We will make a critical use of Premet's results of the centers of complex and modular finite $W$-algebras in \cite{P3}, \cite{Pre3} and \cite{Pre2}, and then of the geometry properties of the Slodowy slices in \cite{Pre1}. Let us introduce the paper as below.

Let $G_\mathbb{C}$ be a simple, simply connected algebraic group over $\bbc$, and ${\ggg}_\mathbb{C}=\text{Lie}(G_\mathbb{C})$. Let $\hat e$ be a nilpotent element in ${\ggg}_\mathbb{C}$. Fix an $\mathfrak{sl}_2$-triple $(\hat e,\hat h,\hat f)$ of $\ggg_\bbc$, and denote by $({\ggg}_\mathbb{C})_{\hat e}:=\text{Ker}(\ad\,\hat e)$ in ${\ggg}_\mathbb{C}$. The linear
operator ad\,$\hat h$ defines a ${\bbz}$-grading ${\ggg}_\mathbb{C}=\bigoplus_{i\in{\bbz}}{\ggg}_\bbc(i)$. Let $(\cdot,\cdot)$  be a scalar multiple of the Killing form of $\ggg_\mathbb{C}$ such that $(\hat e,\hat f)=1$, and define $\chi\in\ggg_\mathbb{C}^*$ by letting $\chi(\hat x)=(\hat e,\hat x)$ for all $\hat x\in{\ggg}_\mathbb{C}$.
Let $\mathfrak{l}_\bbc$ and $\mathfrak{l}'_\bbc$ be  Lagrangian subspaces of $\ggg_\bbc(-1)$ with respect to $(\hat e,[\cdot,\cdot])$ such that $\ggg_\bbc(-1)=\mathfrak{l}_\bbc\oplus\mathfrak{l}'_\bbc$.
Then we define a nilpotent subalgebra $\mmm_\bbc:=\mathfrak{l}'_\bbc\oplus\bigoplus_{i\leqslant-2}\ggg_\bbc(i)$ of $\ggg_\bbc$. Set ${\ppp}_\mathbb{C}:=\bigoplus_{i\geqslant 0}{\ggg}_\mathbb{C}(i)$ and $\tilde{\mathfrak{p}}_\mathbb{C}:={\ppp}_\mathbb{C}\oplus\mathfrak{l}_\bbc$ to be a subalgebra and a subspace of $\ggg_\bbc$, respectively.

Set the generalized Gelfand-Graev module $(Q_\mathbb{C})_\chi:=
U({\ggg}_\mathbb{C})\otimes_{U(\mathfrak{m}_\mathbb{C})}{\bbc}_\chi$, where ${\bbc}_\chi:={\bbc}\hat1_\chi$ is a one-dimensional  $\mathfrak{m}_\mathbb{C}$-module~such that $\hat x.\hat1_\chi=\chi(\hat x)\hat1_\chi$ for all $\hat x\in\mathfrak{m}_\mathbb{C}$.
A finite $W$-algebra over $\bbc$ is defined to be
 $$U({\ggg}_\mathbb{C},\hat e):=(\text{End}_{\ggg_\mathbb{C}}(Q_\mathbb{C})_{\chi})^{\text{op}},$$
which is shown to be the quantization of the Slodowy slice $\cs:=\hat e+\text{Ker}(\ad\,\hat f)$ to the adjoint orbit $\Ad\,G_\bbc.\hat e$ (see \cite{GG} and \cite{Pre1}).

\subsection{}\label{zpcenter}
Let $G$, ${\ggg}$, $\ggg_e$, $\mmm$, $\mathfrak{p}$, $\tilde{\mathfrak{p}}$ and $Q_{\chi}$ be the modular counterparts of $G_\bbc$, $\ggg_\mathbb{C}$, $({\ggg}_\mathbb{C})_{\hat e}$, $\mmm_\bbc$, $\mathfrak{p}_\bbc$, $\tilde{\mathfrak{p}}_\bbc$ and $(Q_\mathbb{C})_{\chi}$, respectively.
With the aid of some admissible ring $A$ (see \S\ref{211}), through the procedure of ``reduction modulo $p$" we can
define the $A$-algebra $U({\ggg}_A,\hat e)$ and then the finite $W$-algebra $\ct({\ggg},e)$ over $\bbk$ by
$$\ct({\ggg},e):=U({\ggg}_A,\hat e)\otimes_A{\bbk},$$where $e:=\hat e\otimes_A1\in\ggg$ is an element obtained from the nilpotent element $\hat e\in\ggg_\mathbb{C}$ by ``reduction modulo $p$".

Let $\mathcal{M}$ be a connected unipotent subgroup of $G$ such that its Lie algebra is equal to $\mmm$ (see \S\ref{222} for the details). By \cite[Lemma 4.4 and Theorem 7.3]{GT} the finite $W$-algebra $\ct(\ggg,e)$ is isomorphic to the $\Ad\,\mathcal{M}$-invariant spaces $Q_\chi^{\mathcal{M}}$ and $(\text{End}_{\ggg}^\mathcal{M}Q_{\chi})^{\text{op}}$. Moreover, the extended finite $W$-algebra over ${\bbk}$ is defined to be $$U({\ggg},e):=(\text{End}_{{\ggg}}Q_{\chi})^{\text{op}}.$$
Let $Z_0(\ggg)$ be the $p$-center of the universal enveloping algebra $U(\ggg)$ which is by definition a subalgebra generated by $x^p-x^{[p]}$ for all $x\in\ggg$.
Given a subspace $V\subseteq\ggg$ we denote by $Z_0(V)$ the subalgebra of $Z_0(\ggg)$ generated by all $x^p-x^{[p]}$ with $x\in V$.
 Set $\mathfrak{a}=\{x\in\tilde{\mathfrak{p}}\mid (x,\text{Ker}(\ad\,f))=0\}$. It follows from \cite[Theorem 2.1(iii)]{Pre3} that  ${U}(\ggg,e)\cong \ct(\ggg,e)\otimes_\bbk Z_0(\mathfrak{a})$ as $\mathds{k}$-algebras.

The $p$-centers of the (extended) finite $W$-algebras are defined as follow. Denote by $Z(\ggg)$  the center of the universal enveloping algebra $U(\ggg)$.
Let $\varphi$ be the natural representation of $\ggg$ over $Q_\chi$, which induces a $\mathds{k}$-algebra homomorphism from $Z(\ggg)$ to $(\text{End}_{{\ggg}}Q_{\chi})^{\text{op}}$ such that $(\varphi(x))(1_\chi)=x.1_\chi\in Q_\chi$ for any $x\in Z(\ggg)$.
Then we have $\bbk$-algebra isomorphisms $\varphi(Z_0(\ggg))\cong\varphi(Z_0(\tilde{\mathfrak{p}}))\cong Z_0(\tilde{\mathfrak{p}})$ (see the proof of \cite[Theorem 2.1]{Pre3}). Now
we identify $Z_0(\tilde{\mathfrak{p}})$ with $\varphi(Z_0(\tilde{\mathfrak{p}}))$ and $\varphi(Z_0(\ggg))$. Correspondingly, $Z_0(\tilde{\mathfrak{p}})$ is naturally playing a role of the $p$-centers of $U(\ggg,e)$ as a counterpart of the $p$-centers in $U(\ggg)$.
Define its invariant subalgebra $Z_0(\tilde{\mathfrak{p}})^{\mathcal{M}}$ under the action of $\Ad\,\mathcal{M}$ as in \S\ref{312}.  According to \cite{GT}, $Z_0(\tilde{\mathfrak{p}})^{\mathcal{M}}$ will play a role similar to a $p$-center of $\ct(\ggg,e)$. For these $p$-centers, one can refer to \cite{GT}, \cite{Pre3} and \cite{Pre2} for more details.

\subsection{}
Let us turn to the case over the field $\bbc$ for a while. It is well-known that the center of $U(\ggg_\bbc)$ is equal to $U(\ggg_\bbc)^{G_\bbc}$, the algebra of $G_\bbc$-invariants for the adjoint action in $U(\ggg_\bbc)$. In the footnote of \cite[Question 5.1]{P3}, Premet showed that the center of finite $W$-algebra $U(\ggg_\bbc,\hat e)$ coincides with the image of $U(\ggg_\bbc)^{G_\bbc}$ in $U(\ggg_\bbc,\hat e)$ under the map $\varphi_\bbc$, where $\varphi_\bbc$ is of the same meaning as $\varphi$ in \S\ref{zpcenter}.

We go back to the case of  positive characteristic.  Denote by $r$ the rank of $\ggg$. Under the assumption that $p$ does not divide the order of the Weyl group associated to a given root system of $\ggg$, by Veldkamp's theorem \cite{Ve} we know that $Z(\ggg)$ is generated by the $p$-center $Z_0(\ggg)$ and the Harish-Chandra center $Z_1(\ggg):=U(\ggg)^G$. Moreover, there exist algebraically independent generators $g_1,\cdots,g_r$ of $U(\ggg)^G$ such that $Z(\ggg)$ is a free $Z_0(\ggg)$-module of rank $p^r$ with a basis consisting of all $g_1^{t_1}\cdots g_r^{t_r}$ with $0\leqslant t_k\leqslant p-1$ for all $k$. 

Now we consider the finite $W$-algebra $\ct(\ggg,e)$  in positive characteristic $p$. Let $Z(\ct)$ be the center of  $\ct(\ggg,e)$, and denote by $\Lambda_k:=\{(i_1,\cdots,i_k)\mid i_j\in\{0,1,\cdots,p-1\}\}$
for $k\in\mathbb{Z}_+$ with $1\leqslant j\leqslant k$. Write $Z_0(\ct):=Z_0(\tilde{\mathfrak{p}})^{\mathcal{M}}$, which is the $p$-center of $\ct(\ggg,e)$ as mentioned above. Note that the image of $Z_1(\ggg)$ under the map $\varphi$ in \S\ref{zpcenter} lies in $\ct(\ggg,e)$, and we denote it by $Z_1(\ct)$. Set $f_i:=\varphi(g_i)$ with $1\leqslant i\leqslant r$, which are all in $Z_1(\ct)$.
As an analogue of the Veldkamp's theorem in the Lie algebra case, our first main result comes as follows:
\begin{theorem}\label{central thm} Under the assumption $p\gg 0$
for ${\bbk}=\overline{\mathbb{F}}_p$, 
we have the following results.
\begin{itemize}
\item[(1)] The $\mathds{k}$-algebra $Z(\ct)$ is generated by  $Z_0(\ct)$ and $Z_1(\ct)$. More precisely, $Z(\ct)$ is a free module of rank $p^r$ over $Z_0(\ct)$ with a basis $f_1^{t_1}\cdots f_r^{t_r}$, where $(t_1,\cdots,t_r)$ runs through $\Lambda_r$.
\item[(2)] The multiplication map $\mu:~Z_0(\ct)\otimes_{Z_0(\ct)\cap Z_1(\ct)} Z_1(\ct)\rightarrow Z(\ct)$ is an isomorphism of
$\mathds{k}$-algebras.
\end{itemize}
\end{theorem}

The proof of the above theorem is dependent on Proposition \ref{Regular pts} involving the normality of the fibers of the restriction to 
Slodowy slices of the adjoint quotient map.  {The proof will be  given in \S\ref{pro} and \S\ref{pro2}.

\subsection{}
We continue to consider the case over $\bbk$.
Let $\Specm(Z(\ct))$ be the spectrum of maximal ideals of $Z(\ct)$. The corresponding Azumaya locus can be defined as \begin{equation*}
\mathcal{A}(\ct(\ggg,e)):=\{\mathbf{m}\in \Specm(Z(\ct))\mid \ct(\ggg,e)_{\mathbf{m}} \mbox{ is Azumaya over }Z(\ct)_{\mathbf{m}}\},
\end{equation*}
where $\ct(\ggg,e)_{\mathbf{m}}$ and $Z(\ct)_{\mathbf{m}}$ denote the localizations of $\ct(\ggg,e)$ and $Z(\ct)$ at $\mathbf{m}$, respectively. Let $V$ be an irreducible $\ct(\ggg,e)$-module, and consider the corresponding central character 
$\zeta_V : Z(\ct) \rightarrow \bbk$.

As we see in \S\ref{ring theory property} and \S\ref{for fintecen}, $\ct(\ggg,e)$ is a prime Noetherian algebra, module-finite over its affine center, then by \cite[\S3]{BGl} the  Azumaya locus can be described in the representation theoretical meanings as:
$$\mathcal{A}(\ct(\ggg,e))
 =\{\ker(\zeta_V)|\dim V=\mbox{the maximal dimension of all irreducible}~ \ct(\ggg,e)\mbox{-modules}\}.$$
The following second main result shows the coincidence of the Azumaya locus with the smooth locus of $\Specm({Z(\ct)})$.

\begin{theorem}\label{Azumaya Thm} 
Let $G_\bbc$ be a simple, simply connected algebraic group over $\bbc$, and $\ggg=\Lie(G)$ for the counterpart $G$ of $G_\bbc$ over $\bbk=\overline{\mathbb{F}}_p$ for $p\gg 0$.  
Then the locus of points in $\Specm(Z(\ct))$ that occur as $Z(\ct)$-annihilators of irreducible $\ct(\ggg,e)$-module of maximal dimension coincides with the smooth locus consisting of smooth  points in $\Specm(Z(\ct))$.
\end{theorem}
The proof of the above theorem will be given in \S\ref{azumaya proof}.
This result reveals a close connection between the geometry of the centers and representations of the maximal dimensions for modular finite $W$-algebras.  Brown-Goodearl \cite{BGl} listed three major examples of  algebras: quantized enveloping algebras,  quantized function algebras at a root of unity, as well as enveloping algebras of reductive Lie algebras in positive characteristic, for which one can parameterize  irreducible representations of the maximal dimensions via the smooth points in the spectrum of maximal ideals of the centers. By Theorem \ref{Azumaya Thm}, modular finite $W$-algebras provide a new example for Brown-Goodearl's list.

\subsection{} The paper is organized as follows. In \S\ref{On the finite} some basics on complex  and modular finite $W$-algebras are recalled.  We also introduce some new observations in \S\ref{ring theory property} and \S\ref{2.5.2}. In \S\ref{fractional  of} we mainly investigate the fractional ring $Q(\ct)$ of $\ct(\ggg,e)$ over the fractional field $\Frac(Z(\ct))$. In \S\ref{On the centers of}, we investigate the centers of $\ct(\ggg,e)$, making some necessary preparation for the proof of Theorem \ref{central thm}. The complete proof of Theorem \ref{central thm} will be postponed till \S\ref{pro}  because it is dependent on Proposition \ref{Regular pts}.

 \S\ref{Azumaya locus} is devoted to the proof of Theorem \ref{Azumaya Thm}. We will first recall some material on Azumaya algebras,  then show that $\ct(\ggg,e)$ is Auslander-regular and Macaulay, with Krull and global dimension less than dim\,$\ggg_e$ (see  Proposition \ref{Auslander Cond}). This enables us to apply Brown-Goodearl's theorem (see Theorem \ref{BGl 3.8}) to our case. According to  Slodowy's  earlier results on transverse slices to nilpotent orbits (see Theorem \ref{premet}
), a main property which we need to establish is on the complement of the smooth locus in the Slodowy slice, which is given in Proposition \ref{Regular pts}. 
In \S\ref{azumaya proof}, we finally show that the codimension of
the complement of Azumaya locus in $\Specm(Z(\ct))$ is at least 2. Applying Brown-Goodearl's theorem to our case, we deduce Theorem \ref{Azumaya Thm}.

 \S\ref{pf for thm 1.1} is devoted to the proof of Theorem \ref{central thm}.  We will first show that $\Specm(\text{gr}(\tilde\cz))$ is a (strict) complete intersection, where $\tilde\cz$ denotes the subalgebra of $Z(\ct)$ generated by $Z_0(\ct)$ and $Z_1(\ct)$, and $\text{gr}(\tilde{\cz})$ the graded algebra of $\tilde{\cz}$ under the Kazhdan grading. Applying Proposition \ref{Regular pts}, we further prove that it is a normal variety, and that $\tilde\cz$ coincides with $Z(\ct)$.


\subsection{}
Throughout the paper we work with the field of complex numbers ${\bbc}$ and the field ${\bbk}=\overline{\mathbb{F}}_p$ for $p\gg 0$. 
The assumption arises from Premet's original arguments on the admissible procedure with $p\in\Pi(A)$ for some admissible ring $A$ when we work with $\bbk$ (see \S\ref{211}; or \cite{GG, GT, Pre3, Pre2}  for more details).

Let ${\bbz}_+$ be the set of all the non-negative integers in ${\bbz}$. For $k\in\mathbb{Z}_+$, define
\begin{equation*}
 \begin{array}{llllll}
 \mathbb{Z}_+^k&:=&\{(i_1,\cdots,i_k)\mid i_j\in\mathbb{Z}_+\},&
 \Lambda_k&:=&\{(i_1,\cdots,i_k)\mid i_j\in\{0,1,\cdots,p-1\}\}
 \end{array}
 \end{equation*}with $1\leqslant j\leqslant k$.
For any $\bbk$-vector space $V$, the dimension of $V$ over $\bbk$ is denoted by $\dim V$.

\section{Finite $W$-algebras}\label{On the finite}
We will maintain  the notations and assumptions  in the introduction.
\subsection{Finite $W$-algebras over the field of complex numbers}
First we recall some facts on finite $W$-algebras over $\mathbb{C}$.

\subsubsection{}\label{211}
Let $T_\mathbb{C}$ be a maximal torus in $G_\mathbb{C}$ and set $\mathfrak{h}_\mathbb{C}=\text{Lie}(T_\mathbb{C})$. Then $\mathfrak{h}_\mathbb{C}$ is a Cartan subalgebra of ${\ggg}_\mathbb{C}$. Let $\Phi$ be the root system of ${\ggg}_\mathbb{C}$ relative to $\mathfrak{h}_\mathbb{C}$. Choose a basis of simple roots $\Delta$ in $\Phi$.
Denote by $\Phi^+$ the corresponding positive system in $\Phi$, and put $\Phi^-:=-\Phi^+$. Let ${\ggg}_\mathbb{C}=\mathfrak{n}^-_\mathbb{C}\oplus\mathfrak{h}_\mathbb{C}\oplus\mathfrak{n}^+_\mathbb{C}$ be the corresponding triangular decomposition of ${\ggg}_\mathbb{C}$.
Choose a Chevalley basis $B=\{\hat e_\gamma\mid\gamma\in\Phi\}\cup\{\hat h_\alpha\mid\alpha\in\Delta\}$ of ${\ggg}_\mathbb{C}$. Let ${\ggg}_{\bbz}$ denote the Chevalley ${\bbz}$-form in ${\ggg}_\mathbb{C}$ and $U_{\bbz}$ the Kostant ${\bbz}$-form of $U({\ggg}_\mathbb{C})$ associated with $B$. Given a ${\bbz}$-module $V$ and a ${\bbz}$-algebra $A$, we write $V_A:=V\otimes_{\bbz}A$.

By the Dynkin-Kostant theory, 
for any given nilpotent element $\hat e\in{\ggg}_{\bbz}$ we can find $\hat f,\hat h\in{\ggg}_\mathbb{Q}$ such that $(\hat e,\hat h,\hat f)$ is an $\mathfrak{sl}_2$-triple in ${\ggg}_\mathbb{C}$. 
Fix $(\cdot,\cdot)$   a scalar multiple of the Killing form of $\ggg_\mathbb{C}$ such that  the Chevalley basis $B$ of ${\ggg}_\mathbb{C}$ takes value in $\mathbb{Q}$, and $(\hat e,\hat f)=1$. Define $\chi\in\ggg_\mathbb{C}^*$ by letting $\chi(\hat x)=(\hat e,\hat x)$ for all $\hat x\in{\ggg}_\mathbb{C}$.

As defined by Premet in \cite[Definition 2.1]{Pre3}, we can choose an {\sl admissible} ring $A$, which is a finitely-generated ${\bbz}$-subalgebra of ${\bbc}$ such that $(\hat e,\hat f)\in A^{\times}$ 
and all bad primes of the root system of ${\ggg}_\mathbb{C}$ and the determinant of the Gram matrix of ($\cdot,\cdot$) relative to a Chevalley basis of ${\ggg}_\mathbb{C}$ are invertible in $A$. 
Denote by $\text{Specm}(A)$ the maximal spectrum of $A$, then for every element $\mathfrak{P}\in\text{Specm}(A)$, the residue field $A/\mathfrak{P}$ is isomorphic to $\mathbb{F}_{q}$, where $q$ is a $p$-power depending on $\mathfrak{P}$. Denote by $\Pi(A)$ the set of all primes $p\in\mathbb{Z}_+$ that occur in this way. By \cite[Lemma 4.4]{Pre3} and its proof, the set $\Pi(A)$ contains almost all primes in $\mathbb{Z}_+$. Denote by ${\ggg}_A$ the $A$-submodule of ${\ggg}$ generated by the Chevalley basis $B$.

Let ${\ggg}_\mathbb{C}(i)=\{\hat x\in{\ggg}_\mathbb{C}\mid[\hat h,\hat x]=i\hat x\}$. Then ${\ggg}_\mathbb{C}=\bigoplus_{i\in{\bbz}}{\ggg}_\mathbb{C}(i)$. By the $\mathfrak{sl}_2$-theory, all subspaces ${\ggg}_\mathbb{C}(i)$ are defined over $\mathbb{Q}$. Also, $\hat e\in{\ggg}_\mathbb{C}(2)$, $\hat f\in{\ggg}_\mathbb{C}(-2)$ and $\hat h\in\mathfrak{h}_\mathbb{C}\subseteq{\ggg}_\mathbb{C}(0)$. There exists a symplectic  bilinear form $\langle\cdot,\cdot\rangle$ on ${\ggg}_\mathbb{C}(-1)$  given by $\langle \hat x,\hat y\rangle:=(\hat e,[\hat x,\hat y])=\chi([\hat x,\hat y])$ for all $\hat x,\hat y\in{\ggg}_\mathbb{C}(-1)$.
We can choose a basis $\{\hat z_1,\cdots,\hat z_{2s}\}$ of ${\ggg}(-1)$ contained in ${\ggg}_\mathbb{Q}$ such that $\langle \hat z_i, \hat z_j\rangle =i^*\delta_{i+j,2s+1}$ for $1\leqslant i,j\leqslant 2s$, where $i^*=\left\{\begin{array}{ll}-1,&\text{if}~1\leqslant i\leqslant s;\\ 1,&\text{if}~s+1\leqslant i\leqslant 2s.\end{array}\right.$
Denote by ${\mathfrak{l}}_\mathbb{C}$ the ${\bbc}$-span of $\hat z_{1},\cdots,\hat z_{s}$ and $\mathfrak{l}'_\mathbb{C}$ the ${\bbc}$-span of $\hat z_{s+1},\cdots,\hat z_{2s}$. 
Set ${\ppp}_\mathbb{C}:=\bigoplus_{i\geqslant 0}{\ggg}_\mathbb{C}(i)$, $\tilde{\ppp}_\mathbb{C}:={\ppp}_\mathbb{C}\oplus{\mathfrak{l}}_\mathbb{C}$ and $\mathfrak{m}_\mathbb{C}:=\bigoplus_{i\leqslant -2}{\ggg}_\mathbb{C}(i)\oplus{\mathfrak{l}}_\mathbb{C}'$. Then $\chi$ vanishes on the derived subalgebra of $\mathfrak{m}_\mathbb{C}$, and ${\ggg}_\mathbb{C}=\tilde{\ppp}_\mathbb{C}\oplus\mathfrak{m}_\mathbb{C}$ as vector spaces.

Write $({\ggg}_\mathbb{C})_{\hat e}$ for the centralizer of $\hat e$ in ${\ggg}_\mathbb{C}$, and denote by $\text{Ker}(\ad\,{\hat f})$ the centralizer of $\hat f$ in $\ggg_\bbc$.
Denote by $\sfd:=\dim({\ggg}_\mathbb{C})_{\hat e}$.
As $\dim{\ggg}_\bbc-\dim({\ggg}_\mathbb{C})_{\hat e}=\sum_{k\geqslant2}
2\dim {\ggg}_\mathbb{C}(-k)+\dim{\ggg}_\mathbb{C}(-1)$, we have  $\dim\mathfrak{m}_\mathbb{C}=\frac{\dim{\ggg}_\bbc-\dim({\ggg}_\mathbb{C})_{\hat e}}{2}=\frac{\dim G_\bbc.\hat e}{2}$.
After enlarging $A$ one can assume that ${\ggg}_A=\bigoplus_{i\in{\bbz}}{\ggg}_A(i)$, and each ${\ggg}_A(i):={\ggg}_A\cap{\ggg}_\mathbb{C}(i)$ is freely generated over $A$ by a basis of the vector space ${\ggg}_\bbc(i)$. 
Then we can assume that $\mathfrak{h}_A:={\ggg}_A\cap\mathfrak{h}_\mathbb{C}$, $\mathfrak{n}^-_A:={\ggg}_A\cap\mathfrak{n}^-_\mathbb{C}$,
$\mathfrak{n}^+_A:={\ggg}_A\cap\mathfrak{n}^+_\mathbb{C}$, $\mathfrak{l}_A:={\ggg}_A\cap{\mathfrak{l}}_\mathbb{C}$, ${\ppp}_A:={\ggg}_A\cap{\ppp}_\mathbb{C}$, $\tilde{\ppp}_A:={\ggg}_A\cap\tilde{\ppp}_\mathbb{C}$,
$\mathfrak{m}_A:={\ggg}_A\cap\mathfrak{m}_\mathbb{C}$ and $({\ggg}_A)_{\hat e}:={\ggg}_A\cap({\ggg}_\mathbb{C})_{\hat e}$ are free $A$-modules and direct summands of ${\ggg}_A$. Moreover, one can assume $\hat e,\hat f\in{\ggg}_A$ after  possibly enlarging $A$, $[\hat e,{\ggg}_A(i)]$ and $[\hat f,{\ggg}_A(i)]$ are direct summands of ${\ggg}_A(i+2)$ and ${\ggg}_A(i-2)$ respectively, and ${\ggg}_A(i+2)=[\hat e,{\ggg}_A(i)]$ for each $i\geqslant -1$ by the $\mathfrak{sl}_2$-theory.

As in \cite[\S4.2-\S4.3]{Pre1}, we can choose a basis $\hat x_1,\cdots,\hat x_\sfd,\hat x_{\sfd+1},\cdots,\hat x_m\in{\ppp}_A$
of the free $A$-module ${\ppp}_A=\bigoplus_{i\geqslant 0}{\ggg}_A(i)$ such that

(a) $\hat x_i\in{\ggg}_A(n_i)$, where $n_i\in{\bbz}_+$ with $1\leqslant i\leqslant m$;

(b) $\hat x_1,\cdots,\hat x_\sfd$ is a basis of the $A$-module $({\ggg}_A)_{\hat e}$;

(c) $\hat x_{\sfd+1},\cdots,\hat x_m\in[\hat f,{\ggg}_A]$.
%
%
%
\subsubsection{}\label{212}
Recall the finite $W$-algebra $U(\ggg_\bbc,\hat e)$ over $\bbc$ is by definition,  equal to $(\End_{\ggg_\mathbb{C}}(Q_\mathbb{C})_\chi)^{\op}$.
Let $(I_\bbc)_\chi$ denote the ideal in $U({\ggg}_\bbc)$ generated by all $\hat x-\chi(\hat x)$ with $\hat x\in\mathfrak{m}_\bbc$. 
Then $ U({\ggg}_\bbc)/(I_\bbc)_\chi\cong(Q_\bbc)_\chi$ as ${\ggg}_\bbc$-modules via the ${\ggg}_\bbc$-module map sending $\hat 1+(I_\bbc)_\chi$ to $\hat1_\chi$.  
Since $(I_\bbc)_\chi$ is stable under the adjoint action of $\mmm_\bbc$, there exists a
canonical isomorphism between $U({\ggg}_\bbc,\hat e)$ and the fixed point algebra $(Q_\bbc)_{\chi}^{\ad\,\mmm_\bbc}$ given by $\hat u\mapsto \hat u(\hat 1_\chi)$ for any $\hat u\in U({\ggg}_\bbc,\hat e)$. {\it In what follows we will often identify $(Q_\bbc)_\chi$ with $U({\ggg}_\bbc)/(I_\bbc)_\chi$, and identify $U({\ggg}_\bbc,\hat e)$ with $(Q_\bbc)_{\chi}^{\ad\,\mmm_\bbc}$.}

For $k\in\mathbb{Z}_+$, define $\mathbb{Z}_+^k:=\{(i_1,\cdots,i_k)\mid i_j\in\mathbb{Z}_+\}$ with $1\leqslant j\leqslant k$, and  set $|\mathbf{i}|:=i_1+\cdots+i_k$.  Given $(\mathbf{a},\mathbf{b})\in{\bbz}^m_+\times{\bbz}^s_+$, let $\hat x^\mathbf{a}\hat z^\mathbf{b}$ denote the monomial $\hat x_1^{a_1}\cdots \hat x_m^{a_m}\hat z_1^{b_1}\cdots \hat z_s^{b_s}$ in $U({\ggg}_\bbc)$. Set $Q_{\chi,A}:=U({\ggg}_A)\otimes_{U(\mathfrak{m}_A)}A_\chi$, where $A_\chi=A\hat1_\chi$. By definition $Q_{\chi,A}$ is a ${\ggg}_A$-stable $A$-lattice in $(Q_\bbc)_{\chi}$ with $\{\hat x^\mathbf{a}\hat z^\mathbf{b}\otimes\hat1_\chi\mid(\mathbf{a},\mathbf{b})\in{\bbz}^m_+\times{\bbz}^s_+\}$
being a free basis. Given $(\mathbf{a},\mathbf{b})\in{\bbz}_+^m\times{\bbz}_+^s$, set
\begin{equation}\label{edegree}
\begin{split}
\text{deg}_{\hat e}(\hat x^\mathbf{a}\hat z^\mathbf{b})=|(\mathbf{a},\mathbf{b})|_{\hat e}:=&\sum_{i=1}^ma_i(n_i+2)+\sum_{i=1}^sb_i=\text{wt}(\hat x^\mathbf{a}\hat z^\mathbf{b})+2\text{deg}_S(\hat x^\mathbf{a}\hat z^\mathbf{b}),
\end{split}
\end{equation}
where $\text{wt}(\hat x^\mathbf{a}\hat z^\mathbf{b}):=\sum_{i=1}^ma_in_i-\sum_{i=1}^sb_i$ and $\text{deg}_S(\hat x^\mathbf{a}\hat z^\mathbf{b}):=|\mathbf{a}|+|\mathbf{b}|$ are the weight and the standard degree of $\hat x^\mathbf{a}\hat z^\mathbf{b}$, respectively. 
By \cite[Theorem 4.6]{Pre1}, the finite $W$-algebra $U({\ggg}_\mathbb{C},\hat e)$ is generated by endomorphisms $\hat\Theta_1,\cdots,\hat\Theta_{\sfd}$ with
\begin{equation}\label{generators}
\hat\Theta_k(\hat1_\chi)=(\hat x_k+\sum_{\mbox{\tiny $\begin{array}{c}|\mathbf{a},\mathbf{b}|_{\hat e}=n_k+2,\\|\mathbf{a}|
+|\mathbf{b}|\geqslant 2\end{array}$}}\hat\lambda^k_{\mathbf{a},\mathbf{b}}\hat x^{\mathbf{a}}
\hat z^{\mathbf{b}}+\sum_{|\mathbf{a},\mathbf{b}|_{\hat e}<n_k+2}\hat\lambda^k_{\mathbf{a},\mathbf{b}}\hat x^{\mathbf{a}}
\hat z^{\mathbf{b}})\otimes\hat1_\chi
\end{equation}
for $1\leqslant  k\leqslant\sfd$, where $\hat\lambda^k_{\mathbf{a},\mathbf{b}}\in\mathbb{Q}$, and $\hat\lambda^k_{\mathbf{a},\mathbf{b}}=0$ if $a_{\sfd+1}=\cdots=a_m=b_1=\cdots=b_s=0$.
Moreover, the monomials $\hat\Theta_1^{a_1}\cdots\hat\Theta_\sfd^{a_\sfd}$ with $(a_1,\cdots,a_\sfd)\in{\bbz}^\sfd_+$  form a basis of the vector space $U({\ggg}_\bbc,\hat e)$. As in \cite{Pre4}, we assume that our admissible ring $A$ contains all $\hat\lambda^k_{\mathbf{a},\mathbf{b}}$'s in \eqref{generators}. For ${\bf a}=(a_1,\cdots,a_\sfd)\in\bbz_+^\sfd$, let
$U({\ggg}_A,\hat e)$ be the $A$-span of the monomials
$\{\hat\Theta_1^{a_1}\cdots\hat\Theta_\sfd^{a_\sfd}\mid {\bf{a}}\in \bbz_+^\sfd\},$ which equals $(\End_{\ggg_A}(Q_{\chi,A}))^{\op}$, an $A$-subalgebra of ${U}(\ggg_\mathbb{C},\hat e)$ (see \cite[(3)]{Pre3} for more details).

Let $\hat Y_1,\cdots, \hat Y_n$ be  homogeneous elements in $\ggg_\bbc$. Assume that $U({\ggg}_\bbc)=\bigcup_{i\in{\bbz}}\text{F}_iU({\ggg}_\bbc)$ is a filtration of $U({\ggg}_\bbc)$, where $\text{F}_iU({\ggg}_\bbc)$ is the ${\bbc}$-span of all $\hat Y_{i_1}\cdots \hat Y_{i_k}$ with $\hat Y_{i_1}\in{\ggg}_\bbc(j_1),\cdots,$\\$\hat Y_{i_k}\in{\ggg}_\bbc(j_k)$ such that $(j_1+2)+\cdots+(j_k+2)\leqslant  i$. This filtration is called {\sl Kazhdan filtration}. It is obvious that the Kazhdan grading of $\hat Y_{i_1}\cdots \hat Y_{i_k}$ in $U(\ggg_\bbc)$ is equal to $\text{wt}(\hat Y_{i_1}\cdots \hat Y_{i_k})+2\text{deg}_S(\hat Y_{i_1}\cdots \hat Y_{i_k})$, where $\text{wt}(\hat Y_{i_1}\cdots \hat Y_{i_k})=j_1+\cdots+j_k$ and $\text{deg}_S(\hat Y_{i_1}\cdots \hat Y_{i_k})=k$ are the weight and the standard degree of $\hat Y_{i_1}\cdots \hat Y_{i_k}$ respectively, which coincides with the $e$-degree of the elements of $(Q_\bbc)_{\chi}$ in \eqref{edegree}. The Kazhdan filtration on $(Q_\bbc)_{\chi}$ is defined by $\text{F}_i(Q_\bbc)_{\chi}:=\pi(\text{F}_iU({\ggg}_\bbc))$ with $\pi:U({\ggg}_\bbc)\twoheadrightarrow U({\ggg}_\bbc)/(I_\bbc)_\chi$ being the canonical homomorphism, which makes $(Q_\bbc)_{\chi}$  a filtered $U({\ggg}_\bbc)$-module. Then there is an induced Kazhdan filtration $\text{F}_i U({\ggg}_\bbc,\hat e)$ on the subspace $U({\ggg}_\bbc,\hat e)=(Q_\bbc)_{\chi}^{\ad\,\mmm_\bbc}$ of $(Q_\bbc)_{\chi}$ such that $\text{F}_j U({\ggg}_\bbc,\hat e)=0$ unless $j\geqslant0$.

Let $S((\ggg_\bbc)_{\hat e})$ denote the symmetric algebra of $(\ggg_\bbc)_{\hat e}$. By the $\mathfrak{sl}_2$-theory,
\begin{equation}\label{eq: decomp p}
\tilde{\ppp}_\mathbb{C}=({\ggg}_\mathbb{C})_{\hat e}\oplus\bigoplus_{i\geqslant2}[\hat f,{\ggg}_\bbc(i)]\oplus{\mathfrak{l}}_\mathbb{C},
 \end{equation}
which gives rise to the projection $\tilde{\ppp}_\mathbb{C}\twoheadrightarrow({\ggg}_\mathbb{C})_{\hat e}$.  Under the above settings, Premet showed in
\cite{Pre1} that there exists an algebra isomorphism under the Kazhdan grading,  
i.e.,
\begin{equation}\label{grse'}
\begin{array}{lcll}
\hat\psi:&\text{gr}(U(\ggg_\bbc,\hat e))&\xrightarrow\sim&S((\ggg_\bbc)_{\hat e})\\ &\text{gr}(\hat \Theta_i)&\mapsto&\hat x_i
\end{array}
\end{equation}for $1\leqslant i\leqslant\sfd$.

\subsection{Finite $W$-algebras in positive characteristic}
\subsubsection{}\label{221}
In the subsequent arguments, we mainly make use of the method of ``reduction modulo $p$" to study the counterparts of all the above over the algebraic closured field ${\bbk}=\overline{\mathbb{F}}_p$ of positive characteristic $p\in\Pi(A)$.
The bilinear form $(\cdot,\cdot)$ 
induces a bilinear form on the Lie algebra ${\ggg}\cong{\ggg}_A\otimes_A{\bbk}$, which 
will still be denoted 
by $(\cdot,\cdot)$. If we denote by $G$ the algebraic ${\bbk}$-group associated with the distribution algebra $U=U_{\bbz}\otimes_{\bbz}{\bbk}$, then ${\ggg}=\text{Lie}(G)$ (cf. \cite[\S2.1]{Jan2}). 
For $\hat x\in{\ggg}_A$, set $x:=\hat x\otimes_A1$ to be the corresponding element in ${\ggg}$. Denote by $e=\hat e\otimes_A1,~f=\hat f\otimes_A1$ and $h=\hat h\otimes_A1$ in ${\ggg}$.

Consider the $p$-center $Z_0(\ggg)$ of $U(\ggg)$. There is a $G$-equivariant $\bbk$-algebra isomorphism
$
\varsigma:S(\ggg)^{(1)}\xrightarrow\sim Z_0({\ggg})
$
determined by sending $x\mapsto x^p-x^{[p]}$ for $x\in\ggg$, where the superscript ${}^{(1)}$ denotes the Frobenius twist.
Since the Frobenius map of ${\bbk}$ is bijective, this enables us to identify the maximal spectrum $\text{Specm}(Z_0({\ggg}))$ of $Z_0({\ggg})$ with $({\ggg}^*)^{(1)}$.
For any $\xi\in{\ggg}^*$ we denote by $J_\xi$ the two-sided ideal of $U({\ggg})$ generated by the central elements $\{x^p-x^{[p]}-\xi(x)^p\mid x\in{\ggg}\}$. The quotient algebra $U_\xi({\ggg}):=U({\ggg})/J_\xi$ is called the reduced enveloping algebra with $p$-character $\xi$.

\subsubsection{}\label{222}
For $i\in{\bbz}$, set ${\ggg}(i):={\ggg}_A(i)\otimes_A{\bbk}$ and $\mathfrak{m}:=\mathfrak{m}_A\otimes_A{\bbk}$. Then $\mathfrak{m}$ is a restricted subalgebra of $\mathfrak{g}$. Denote by $\mathfrak{h}:=\mathfrak{h}_A\otimes_A{\bbk}$, $\mathfrak{n}^-:=\mathfrak{n}^-_A\otimes_A{\bbk}$, $\mathfrak{n}^+:=\mathfrak{n}^+_A\otimes_A{\bbk}$, $\mathfrak{l}:=\mathfrak{l}_A\otimes_A{\bbk}$, $\mathfrak{p}:=\mathfrak{p}_A\otimes_A{\bbk}$,
$\tilde{\mathfrak{p}}:=\tilde{\mathfrak{p}}_A\otimes_A{\bbk}$ and
$\mathfrak{\ggg}_e:=({\ggg}_A)_{\hat e}\otimes_A{\bbk}$ as in \S\ref{211}, respectively.

By our assumptions at the end of \S\ref{211} and the procedure of ``modulo $p$ reduction", the elements $x_1,\cdots,x_\sfd$ form a basis of the centralizer ${\ggg}_e$ of $e$ in ${\ggg}$. 
Set $Q_{\chi}:=U({\ggg})\otimes_{U(\mathfrak{m})}{\bbk}_\chi$, 
then $Q_{\chi}\cong Q_{\chi,A}\otimes_A{\bbk}$ as ${\ggg}$-modules. 
Let $I_{\chi}$ denote the left ideal of $U({\ggg})$ generated by all $x-\chi(x)$ with $x\in\mathfrak{m}$. Then we further have $U(\ggg)/I_\chi\cong Q_{\chi}$ as ${\ggg}$-modules via the ${\ggg}$-module map sending $1+I_\chi$ to $1_\chi$.

Recall a finite $W$-algebra $\ct(\ggg,e)$ is by definition, equal to  $U({\ggg}_A,\hat e)\otimes_A{\bbk}$.
Clearly $\ct({\ggg},e)$ has a ${\bbk}$-basis consisting of all monomials $\Theta_1^{a_1}\cdots\Theta_{\sfd}^{a_{\sfd}}$, where $\Theta_i:=\hat\Theta_i\otimes1\in U({\ggg}_A,\hat e)\otimes_A{\bbk}$ with
\begin{equation}\label{generators'}
\Theta_k(1_\chi)=(x_k+\sum_{\mbox{\tiny $\begin{array}{c}|\mathbf{a},\mathbf{b}|_{e}=n_k+2,\\|\mathbf{a}|
+|\mathbf{b}|\geqslant 2\end{array}$}}\lambda^k_{\mathbf{a},\mathbf{b}}x^{\mathbf{a}}
z^{\mathbf{b}}+\sum_{|\mathbf{a},\mathbf{b}|_{e}<n_k+2}\lambda^k_{\mathbf{a},\mathbf{b}} x^{\mathbf{a}}z^{\mathbf{b}})\otimes1_\chi
\end{equation}
for $1\leqslant  k\leqslant\sfd$, where $\lambda^k_{\mathbf{a},\mathbf{b}}\in\bbk$, and $\lambda^k_{\mathbf{a},\mathbf{b}}=0$ if $a_{\sfd+1}=\cdots=a_m=b_1=\cdots=b_s=0$.
Let $\text{gr}(\ct(\ggg,e))$ denote the graded algebra of $\ct(\ggg,e)$ under the Kazhdan grading, and $S(\ggg_e)$ the symmetric algebras of $\ggg_e$. Let $\tilde{\ppp}\twoheadrightarrow{\ggg}_{e}$ be the projection along the decomposition $\tilde{\ppp}={\ggg}_{e}\oplus\bigoplus_{i\geqslant2}[f,{\ggg}(i)]\oplus\mathfrak{l}$  (see \eqref{eq: decomp p}). By the same discussion as in \eqref{grse'} there also exists an isomorphism between $\mathds{k}$-algebras
\begin{equation}\label{grse}
\begin{array}{lcll}
\bar\psi:&\text{gr}(\ct(\ggg,e))&\xrightarrow\sim&S(\ggg_e)\\ &\text{gr}(\Theta_i)&\mapsto&x_i
\end{array}
\end{equation}for $1\leqslant i\leqslant\sfd$.

Now we introduce two more equivalent definitions of the finite $W$-algebra $\ct({\ggg},e)$.
 At first, we describe $\mathcal{M}$ mentioned in the introduction, which is by definition, the connected unipotent of $G$ such that $\text{Ad}\,\mathcal{M}$ is generated by all linear operators $\text{exp(ad}\,x)$ with $x\in\mmm$.
It follows from \cite[Lemma 4.1]{GT} that $I_\chi$ is stable under the adjoint action of $\mathcal{M}$. Then
the adjoint action of $\mathcal{M}$ on $Q_\chi$ given by $g.(u+I_\chi)=(g.u)+I_\chi$
for $g\in\mathcal{M}$ and $u\in U(\ggg)$ descends to an adjoint action on $\text{gr}(Q_\chi)$. Moreover, the $\mathcal{M}$-invariant space \begin{equation*}\label{W2}
Q_\chi^{\mathcal{M}}:=\{u+I_\chi\in Q_\chi\mid g.u+I_\chi=u+I_\chi~\text{for all}~g\in\mathcal{M}\}
\end{equation*}is readily an algebra, which is introduced as the finite $W$-algebra over $\mathds{k}$
in \cite[Definition 4.3]{GT}.
Since $Q_\chi$ is a locally finite $\mathcal{M}$-module under the adjoint action, the differential of this $\mathcal{M}$-module structure coincides with the $\mmm$-module structure on $Q_\chi$ given by the adjoint action. Then the invariant space $Q_{\chi}^{\text{ad}\,{\mmm}}$
inherits an algebra structure from $U(\ggg)$ with the $\mathcal{M}$-invariants
$Q_\chi^{\mathcal{M}}$ embedded as a subalgebra (see \cite[Lemma 4.2]{GT} for more details).
On the other hand, 
the adjoint action of $\mathcal{M}$ on $Q_\chi$ induces an action on $(\text{End}_{\ggg}Q_{\chi})^{\text{op}}$ by $(g.f)(u+I_\chi)=g.(f(g^{-1}.u+I_\chi))$ for $g\in\mathcal{M}$, $f\in\text{End}_{\ggg}Q_{\chi}$, $u\in U(\ggg)$, and the invariant subalgebra is denoted by $(\text{End}_{\ggg}^\mathcal{M}Q_{\chi})^{\text{op}}$.
It follows from \cite[Lemma 4.4 and Theorem 7.3]{GT} that
\begin{equation}\label{equiW}
\ct(\ggg,e)\cong Q_\chi^{\mathcal{M}}\cong (\text{End}_{\ggg}^\mathcal{M}Q_{\chi})^{\text{op}}
\end{equation} as $\mathds{k}$-algebras. {\it We will identify $\ct(\ggg,e)$ with $Q_\chi^{\mathcal{M}}$ and $(\text{End}_{\ggg}^\mathcal{M}Q_{\chi})^{\text{op}}$ throughout the paper.}

Recall that the extended finite $W$-algebra associated to $\chi$ over $\mathds{k}$ is by definition, equal to $(\text{End}_{{\ggg}}Q_{\chi})^{\text{op}}$. By the same discussion as the case over $\bbc$ in \S\ref{212}, there exists a $\bbk$-algebras isomorphism
$(\text{End}_{\ggg}Q_{\chi})^{\text{op}}\cong Q_{\chi}^{\text{ad}\,{\mmm}}$ by sending $u$ to $u(1_\chi)$ for any $u\in U({\ggg},e)$.
{\it From now on, we will identify $(\text{End}_{{\ggg}}Q_{\chi})^{\text{op}}$ with $Q_{\chi}^{\text{ad}\,{\mmm}}$.} Moreover, it is notable that the finite $W$-algebra $\ct({\ggg},e)$  can be naturally identified with a subalgebra of the extended finite $W$-algebra $U({\ggg},e)$ over ${\bbk}$ by  definition. The standard grading on $U(\ggg)$ and the Kazhdan filtrations on $\ct({\ggg},e)$ and $U({\ggg},e)$ can also be defined as in \S\ref{212}.
\subsection{The ring-theoretic property of finite $W$-algebras}\label{ring theory property}
To discuss the related topics on finite $W$-algebras and their subalgebras, we first need the following observation.
\begin{lemma} \label{Noe Pri}The following statements hold.
\begin{itemize}
\item[(1)] 
Both $\text{gr}(U(\ggg_\bbc,\hat e))$ and $\text{gr}(\ct(\ggg,e))$ are unique factorization domains; 
\item[(2)] Both $U(\ggg_\bbc,\hat e)$ and $\ct(\ggg,e)$ are Noetherian rings;
\item[(3)] Both $U(\ggg_\bbc,\hat e)$ and $\ct(\ggg,e)$ are prime rings, which do not contain any zero-divisor.
\end{itemize}
\end{lemma}
\begin{proof}
First recall the isomorphisms in \eqref{grse'} and \eqref{grse}, which show that the gradation of $U(\ggg_\bbc,\hat e)$ and $\ct(\ggg,e)$ are isomorphic to polynomial algebras. Then the standard filtration arguments work, as
these properties hold for the associated graded algebras.
The detailed proof will be omitted here.
\end{proof}

\subsection{The $p$-centers of finite $W$-algebras and extended finite $W$-algebras}\label{p-center}
Let $\varphi$ be the natural representation of $\ggg$ over $Q_\chi$, which induces a $\mathds{k}$-algebras homomorphism from $Z(\ggg)$ to $U(\ggg,e)$ such that
\[\begin{array}{lcll}
\varphi:&Z(\ggg)&\rightarrow&U(\ggg,e)\\ &x&\mapsto&l_x
\end{array}
\]where $l_x(1_\chi)=x.1_\chi\in Q_\chi$ for any $x\in Z(\ggg)$.  Recall that  $U(\ggg,e)=(\text{End}_{\ggg}Q_{\chi})^{\text{op}}$ which is isomorphic to $Q_{\chi}^{\text{ad}\,{\mmm}}$ by our earlier discussion. For any $x\in Z(\ggg)$ we can also consider $\varphi(x)$ as the  element $(\varphi(x))(1_\chi)$ in $Q_{\chi}^{\text{ad}\,{\mmm}}$ and do not distinguish them from now on.
The map $\varphi$ plays a critical rule in this paper.
\subsubsection{}\label{pcen}
We first look at the $p$-center of the extended finite $W$-algebra $U(\ggg,e)$ introduced in \S\ref{zpcenter}.
Note that
$Z_0(\tilde{\mathfrak{p}})\cap\ker(\varphi)=\{0\}$ by the PBW theorem. So
$\varphi(Z_0(\ggg))=\varphi(Z_0(\tilde{\mathfrak{p}}))\cong Z_0(\tilde{\mathfrak{p}})$ as $\bbk$-algebras.
{\it From now on we identify $Z_0(\tilde{\mathfrak{p}})$ with $\varphi(Z_0(\tilde{\mathfrak{p}}))$ and $\varphi(Z_0(\ggg))$ in the paper}.

\subsubsection{}\label{312}
Now we describe  the $p$-center of the finite $W$-algebra $\ct(\ggg,e)$ mentioned  in \S\ref{zpcenter} (see \cite[Remark 2.1]{Pre2} and \cite[\S8]{GT} for more details).
Recall that the associated graded algebra of $Q_\chi$ under the Kazhdan grading is $\text{gr}(Q_\chi)=S(\ggg)/\text{gr}(I_\chi)$, and by the PBW theorem we have that $S(\ggg)=S(\tilde{\mathfrak{p}})\oplus\text{gr}(I_\chi)$. As $\text{pr}: S(\ggg)\rightarrow S(\tilde{\mathfrak{p}})$ is the projection along this direct sum decomposition, this restricts to an isomorphism
$\text{gr}(Q_{\chi})\cong S(\tilde{\mathfrak{p}})$.

The adjoint action of $\mathcal{M}$ (defined in \S\ref{222}) on $Q_\chi$ descends to an adjoint action on $\text{gr}(Q_\chi)$ and this gives a twisted action of $\mathcal{M}$ on $S(\tilde{\mathfrak{p}})$ defined by $\text{tw}(g)\cdot f:=\text{pr}(g.f)$
for $g\in\mathcal{M}$ and $f\in S(\tilde{\mathfrak{p}})$, where $g.f$ denotes the usual adjoint action of $g$ on $f$ in $S(\ggg)$. We write $S(\tilde{\mathfrak{p}})^{\mathcal{M}}$ for the invariants with respect to this action.

On the other hand, 
we can identify $S(\ggg)$ with the algebra $\mathds{k}[\ggg^*]$ of regular functions on the affine variety $\ggg^*$.
Let $\mathfrak{m}^\perp$ denote the set of all linear functions on $\ggg$ vanishing on $\mmm$, i.e., ${\mmm}^\perp:=\{f\in\ggg^*\mid f(\mmm)=0\}$.
Then $\text{gr}(I_\chi)$ is 
the ideal of all functions in $\mathds{k}[\ggg^*]$ vanishing on the closed subvariety $\chi+{\mmm}^\perp$ of $\ggg^*$. 
In this way, we have identified $\text{gr}(Q_{\chi})\cong\mathds{k}[\chi+\mathfrak{m}^\perp]$, and then $S(\tilde{\mathfrak{p}})\cong\mathds{k}[\chi+\mathfrak{m}^\perp]$.

As we identify $S(\tilde{\mathfrak{p}})$ with $\mathds{k}[\chi+\mathfrak{m}^\perp]$, we may regard the $\mathcal{M}$-algebra $Z_0(\tilde{\mathfrak{p}})$ as the coordinate algebra of the Frobenius twist $(\chi+\mathfrak{m}^\perp)^{(1)}\subseteq(\ggg^*)^{(1)}$ of $\chi+\mathfrak{m}^\perp$, where the natural action $\mathcal{M}$ on $(\chi+\mathfrak{m}^\perp)^{(1)}$ is a Frobenius twist of the coadjoint action of $\mathcal{M}$ on $\chi+\mathfrak{m}^\perp$. As $\mathds{k}[\chi+\mathfrak{m}^\perp]^{\mathcal{M}}\cong\mathds{k}[\chi+\kappa(\text{Ker}(\ad\,f))]$ by \cite[Lemma 3.2]{Pre3}  (where the map $\kappa$ is the Killing isomorphism from $\ggg$ to $\ggg^*$ taking $x$ to $(x,\cdot)$), write $Z_0(\tilde{\mathfrak{p}})^{\mathcal{M}}:=\mathds{k}[(\chi+\mathfrak{m}^\perp)^{(1)}]^{\mathcal{M}}
\cong\mathds{k}[(\chi+\kappa(\text{Ker}(\ad\,f)))^{(1)}]$, the function algebra on the Frobenius twist of $\kappa(\cs)$ with $\cs:=e+\text{Ker}(\ad\,f)$ being the Slodowy slice.

Recall the isomorphism $\varsigma:S(\ggg)^{(1)}\xrightarrow\sim Z_0({\ggg})$ from \S\ref{221}. Write $\mathfrak{m}_\chi$ for the ideal of $U(\mathfrak{m})$ generated by all $x-\chi(x)$ with $x\in\mathfrak{m}$.
We write $I_p:=\varsigma(\mathfrak{m}_\chi^{(1)})Z_0(\ggg)$ for $\varsigma(\mathfrak{m}_\chi^{(1)}):=\{x^p-x^{[p]}-\chi(x)^p\mid x\in\mathfrak{m}\}$.
Since the group $\mathcal{M}$ preserves the left ideal $I_\chi$ and $\varsigma$ is $G$-equivariant, then $\mathcal{M}$ acts on both $U(\ggg,e)\cong Q_{\chi}^{\text{ad}\,{\mmm}}$ and $Z_0(\ggg)/I_p$.
 Let $U(\mathfrak{g},e)^{\mathcal{M}}$ and $(Z_0(\ggg)/I_p)^{\mathcal{M}}$  denote the fixed point subspace of $U(\mathfrak{g},e)$ and $Z_0(\ggg)/I_p$ under the action of $\text{Ad}\,\mathcal{M}$, respectively.

 On the other hand, since $Z_0(\tilde{\mathfrak{p}})\cong\varphi(Z_0(\ggg))\cong Z_0(\ggg)/I_p$ by definition, then we have  $Z_0(\tilde{\mathfrak{p}})^{\mathcal{M}}
 \cong (Z_0(\ggg)/I_p)^{\mathcal{M}}$ as $\bbk$-algebras. {\it We will identify $Z_0(\tilde{\mathfrak{p}})^{\mathcal{M}}$ with $(Z_0(\ggg)/I_p)^{\mathcal{M}}$ in the paper.}
In virtue of \eqref{equiW}, $Z_0(\tilde{\mathfrak{p}})^{\mathcal{M}}$ is a subalgebra of the finite $W$-algebra $\ct(\ggg,e)=U(\mathfrak{g},e)^{\mathcal{M}}$. Moreover, in \cite[Remark 2.1]{Pre2} Premet showed that
\begin{gather}\label{twoequation}
Z_0(\tilde{\mathfrak{p}})^{\mathcal{M}}=Z_0(\tilde{\mathfrak{p}})\cap \ct(\mathfrak{g},e).
\end{gather}
\begin{defn}\label{pcenterW}
The $p$-center of the finite $W$-algebra $\ct(\ggg,e)$ is defined as the invariant subalgebra $Z_0(\tilde{\mathfrak{p}})^{\mathcal{M}}(\cong (Z_0(\ggg)/I_p)^{\mathcal{M}}$).
\end{defn}

For any $\mathbf{a}=(a_1,\cdots,a_m)\in\mathbb{Z}_+^m$ and $\mathbf{b}=(b_1,\cdots,b_s)\in\mathbb{Z}_+^s$, set $$x^\mathbf{a}z^{\mathbf{b}}:=x_1^{a_1}\cdots x_m^{a_m}z_1^{b_1}\cdots z_s^{b_s}.$$
Using \eqref{generators'}, Goodwin-Topley obtained descriptions of PBW generators of
$Z_0(\tilde{\mathfrak{p}})^{\mathcal{M}}$ in \cite[(8.2)]{GT} as
\begin{equation}\label{xkzp}
\Phi(x_k)=(\varsigma(x_k)+\sum_{\mbox{\tiny $\begin{array}{c}|\mathbf{a},\mathbf{b}|_{e}=n_k+2,\\|\mathbf{a}|+|\mathbf{b}|
\geqslant2\end{array}$}}\lambda^k_{\mathbf{a},\mathbf{b}}\varsigma(x^{\mathbf{a}}z^{\mathbf{b}})+I_p)\otimes1_\chi
\end{equation}
for $1\leqslant k\leqslant\sfd$ with $\lambda^k_{\mathbf{a},\mathbf{b}}\in\bbk$.

Keeping in mind that $x^{[p]}\in\ggg(pi)$ whenever $x\in\ggg(i)$ for all $i\in\mathbb{Z}$
. Then in the graded algebra associated with the Kazhdan-filtered algebra $U(\ggg)$ we have $\text{deg}_e(x_i^p-x_i^{[p]})=\text{deg}_e(x_i^p)$ for $1\leqslant i\leqslant \sfd$.  So it readily follows from \eqref{grse} and \eqref{xkzp} that the following restriction of $\bar\psi$, denoted by  the same notation, gives rise to an isomorphism
 \begin{equation}\label{grp}
\bar\psi:\text{gr}(Z_0(\tilde{\mathfrak{p}})^{\mathcal{M}})\xrightarrow\sim\bbk[x_1^p,\cdots,x_\sfd^p].
\end{equation}

\subsubsection{}\label{for fintecen}
Let $\mathfrak{a}:=\{x\in\tilde{\mathfrak{p}}\mid (x,\text{Ker}(\ad\,f))=0\}$ be a subspace of $\tilde{\mathfrak{p}}$.
Thanks to the projection $\tilde{\ppp}\twoheadrightarrow{\ggg}_{e}$ (see \eqref{eq: decomp p}), the following results can be obtained by applying \cite[(8.2), (7.4) and (8.3)]{GT}.

\begin{lemma}(\cite{GT})\label{pmge}
There exist isomorphisms between $\mathds{k}$-algebras:
\begin{itemize}
\item[(1)] $ Z_0(\tilde{\mathfrak{p}})^{\mathcal{M}}\cong Z_0(\ggg_e)$;
\item[(2)] $Z_0(\widetilde{\mathfrak{p}})^{\mathcal{M}}\otimes_{\mathds{k}}Z_0(\mathfrak{a})\cong Z_0(\tilde{\mathfrak{p}})$ under the multiplication map.
\end{itemize}
\end{lemma}
\noindent In particular, since $Z_0(\ggg_e)$ is a polynomial algebra in $\dim\ggg_e$ variables, Lemma \ref{pmge}(1) entails that $Z_0(\tilde{\mathfrak{p}})^{\mathcal{M}}$ also has this property.

In \cite[Theorem 2.1]{Pre3} and \cite[Remark 2.1]{Pre2}, Premet introduced the following transition property between finite $W$-algebras and their extended counterparts.
\begin{theorem} (\cite{Pre3, Pre2})\label{prem}
The following hold.
\begin{itemize}
\item[(1)] The algebra ${U}(\ggg,e)$ is generated by its subalgebras $\ct(\ggg,e)$ and $Z_0(\tilde{\mathfrak{p}})$;
\item[(2)] ${U}(\ggg,e)\cong \ct(\ggg,e)\otimes_\bbk Z_0(\mathfrak{a})$ as $\mathds{k}$-algebras;
\item[(3)] $U(\ggg,e)$ is a free  module over $Z_0(\tilde{\mathfrak{p}})$ of rank $p^{\dim\ggg_e}$;
\item[(4)] $\ct(\ggg,e)$ is a free  module over $Z_0(\tilde{\mathfrak{p}})^{\mathcal{M}}$ of rank $p^{\dim  \ggg_e}$.
\end{itemize}
\end{theorem}

\subsection{The reduced $W$-algebras}
\subsubsection{}\label{251}
Given a linear function $\eta\in\chi+{\mmm}^\perp$, set the ${\ggg}$-module $Q_{\chi}^\eta:=Q_{\chi}/J_\eta Q_{\chi},$ where $J_\eta$ is the ideal of $U({\ggg})$ generated by all $\{x^p-x^{[p]}-\eta(x)^p\mid  x\in{\ggg}\}$. Evidently $Q_{\chi}^\eta$ is a ${\ggg}$-module~with $p$-character $\eta$, and
there exists a ${\ggg}$-module isomorphism
$
Q_{\chi}^\eta\cong U_\eta({\ggg})\otimes_{U_\eta({\mmm})}\bbk_\chi
$ by \cite[Lemma 2.2(i)]{Pre3}.
By  definition, the restriction of $\eta$ on $\mmm$ coincides with that of $\chi$ on ${\mmm}$.

\begin{defn}\label{reduced W}
Define a reduced extended $W$-algebra $U_\eta({\ggg},e)$ associated to $\ggg$ with $p$-character $\eta\in\chi+{\mmm}^\perp$ by $$U_\eta({\ggg},e):=(\text{End}_\ggg Q_\chi^\eta)^{\text{op}}.$$
\end{defn}

Next we will introduce other equivalent definitions of reduced extended  $W$-algebras.
In \cite[Theorem 2.2(b)]{Pre3}, Premet introduced the $\mathds{k}$-algebra
\begin{equation*}
U^\circ_\eta({\ggg},e):=U({\ggg},e)\otimes_{Z_0(\tilde{\mathfrak{p}})}\mathds{k}_\eta,
\end{equation*}
where $\mathds{k}_\eta:=\mathds{k}1_\eta$ is a one-dimensional $Z_0(\tilde{\mathfrak{p}})$-module such that $(x^p-x^{[p]}).1_\eta=\eta(x)^p1_\eta$ for all $x\in\tilde{\mathfrak{p}}$. Equivalently, if we let $H_\eta$ be the ideal of $U(\ggg,e)$ generated all $\varphi(x^p-x^{[p]}-\eta(x)^p)$ with $x\in\tilde{\mathfrak{p}}$, then
\begin{equation}\label{quo}
 U^\circ_\eta({\ggg},e)=U(\ggg,e)/H_\eta.
\end{equation}
Moreover, Premet also showed that the canonical projection
$Q_\chi\rightarrow Q_\chi^\eta=Q_{\chi}/J_\eta Q_{\chi}$
gives rise to an isomorphism between $\bbk$-algebras
\begin{equation}\label{reduced W3}
\psi_\eta:~U^\circ_{\eta}({\ggg},e)\cong U_\eta({\ggg},e).
\end{equation}

In \S\ref{312} we have defined
$Z_0(\tilde{\mathfrak{p}})^{\mathcal{M}}$ as  $\mathds{k}[(\chi+\kappa(\text{Ker}(\ad\,f)))^{(1)}]$.
For any $\eta\in\chi+\kappa(\text{Ker}(\ad\,f))$, we write $K_\eta$ for the corresponding maximal ideal of $(Z_0(\ggg)/I_p)^{\mathcal{M}}(\cong Z_0(\tilde{\mathfrak{p}})^{\mathcal{M}})$. In \cite[Definition 8.5]{GT}, Goodwin-Topley introduced the definition of reduced $W$-algebras as follows.
\begin{defn}\label{redu}
The reduced $W$-algebra $\ct_\eta({\ggg},e)$ associated to $\ggg$ with $p$-character $\eta\in\chi+{\mmm}^\perp$ is defined by \begin{equation*}
\ct_\eta({\ggg},e):=\ct({\ggg},e)/\ct({\ggg},e)K_\eta.
\end{equation*}\end{defn}
It is immediate that the category of $\ct_\eta(\ggg,e)$-modules can be naturally regarded as a subcategory of $\ct(\ggg,e)$-modules. The objects in the former are called $\ct(\ggg,e)$-modules of $p$-character $\eta$.
Moreover, Goodwin-Topley proved in \cite[Proposition 8.7]{GT} that
\begin{equation*}\label{reducediso}
U_\eta({\ggg},e)\cong\ct_\eta({\ggg},e)
\end{equation*}as $\mathds{k}$-algebras. Therefore, {\it we will take $\ct_\eta({\ggg},e)$ as the $\mathds{k}$-algebras $U_\eta({\ggg},e)$ and $U^\circ_\eta({\ggg},e)$ in the paper, and call them uniformly as the reduced $W$-algebras.}
\subsubsection{}\label{2.5.2}
Denote by $U_\eta(\ggg)$-mod and by $U_\eta(\ggg,e)$-mod the categories of finite-dimensional modules over $U_\eta(\ggg)$ and over $U_\eta(\ggg,e)$, respectively. We will establish an equivalence between them.
Let $N_{{\mmm}}$ denote the Jacobson radical of $U_\eta(\mmm)$, which is the left ideal of codimensional one in $U_\eta(\mmm)$ generated by all $\langle x-\eta(x)\mid x\in\mmm\rangle$, and write $I_{\mmm}:=U_\eta({\ggg})N_{\mmm}$.
Given a left $U_{\eta}({\ggg})$-module $M$, define
$$M^{\mmm}:=\{v\in M\mid I_{\mmm}.v=0\}.$$
Let $U_\eta({\ggg})^{\text{ad}\,{\mmm}}$ denotes the centralizer of $\mmm$ in $U_\eta({\ggg})$ under the adjoint action, then we have
\begin{lemma}\label{redum}
\begin{equation*}
U_{\eta}({\ggg},e)\cong(Q^\eta_\chi)^{\text{ad}\,{\mmm}}\cong U_\eta({\ggg})^{\text{ad}\,{\mmm}}/U_\eta({\ggg})^{\text{ad}\,{\mmm}}\cap I_{\mmm}.
\end{equation*}
\end{lemma}
\begin{proof}
The first isomorphism in the lemma is just \cite[Lemma 2.2 (ii)]{Pre3}, but the second one is not straightforward. We will give the detailed proof as below.

Since $\eta|_{\mmm}=\chi|_{\mmm}$ by definition, the same method as in the proof of \cite[Theorem 2.3(iv)]{Pre1} can be applied. Explicitly speaking, for any $U_\eta({\ggg})$-module $M$, let $\mathcal{V}_{\ggg}(M)$ be the support variety of $M$ consists $0$ and all those  $x\in\mathcal{N}_p(\ggg):=\{x^{[p]}=0\mid x\in\ggg\}$ for which $M$ is not a free $U_\eta(x)$-module, where $U_\eta(x)$ denotes the subalgebra with $1$ of $U_\eta({\ggg})$ generated by $x\in\ggg$.
One knows that $M$ is a projective $U_\eta({\ggg})$-module if and only if $\mathcal{V}_{\ggg}(M)=\{0\}$ (see \cite[Proposition 6.2]{FP}).

Let $E_1,\cdots,E_t$ be representatives of the isomorphism classes of simple $U_\eta({\ggg})$-modules, and define $\mathcal{V}_{\ggg}(\eta):=\bigcup_{i=1}^t\mathcal{V}_{\ggg}(E_i)$. 
Then the restricted subalgebra $\mmm$ of $\ggg$ is $\eta$-admissible in the meaning of \cite[Definition 2.3]{Pre1}. 
Combining this with \cite[Proposition 7.1(a)]{FP} and \cite[Proposition 2.2]{P0} we can observe that the support variety of the adjoint $U_0({\mmm})$-module
$U_\eta({\ggg})$ equals $\mathcal{V}_{\ggg}(U_\eta({\ggg})_{\text{ad}})\cap\mmm=\mathcal{V}_{\ggg}(\eta)\cap\mmm=\{0\}$. So the adjoint
$U_0({\mmm})$-module $U_\eta({\ggg})$ is projective, hence free (for the algebra $U_0({\mmm})$ is local). As $N_{\mmm}$ is a two-sided ideal of $U_\eta({\mmm})$ the left ideal $I_{\mmm}$ is $(\text{ad}\,\mmm)$-stable. It follows from \cite[Theorem 1.3]{Skry} that every finite-dimensional $U_\eta(\ggg)$-module is $U_\eta(\mmm)$-free, so is the left $U_\eta(\mmm)$-module $\bar U:=U_\eta({\ggg})/I_{\mmm}$. Given $x\in\mmm$ and $u\in U_\eta(\ggg)$ one has
\begin{equation*}
x(u+I_{\mmm})=(\eta(x)u+[x-\eta(x),u])+I_{\mmm}=(\eta(x)u+[x,u])+I_{\mmm}.
\end{equation*}
Then for any $x\in\mathcal{N}_p(\ggg)\cap\mmm$, the endomorphism $\text{ad}\,x$ acts on $\bar U$ as a direct sum of Jordan blocks of length $p$. This shows that the support variety of the adjoint $U_0(\mmm)$-module $\bar U$ is zero. So the adjoint $U_0(\mmm)$-module $\bar U$ is projective, hence free. Thus the short exact sequence of $(\text{ad}\,\mmm)$-modules
\begin{equation*}
0\rightarrow I_{\mmm}\rightarrow U_\eta(\ggg)\rightarrow\bar U\rightarrow0
\end{equation*}splits.
In other words, there is an ad\,$\mmm$-module $V\subseteq U_\eta(\ggg)$ such that $U_\eta(\ggg)\cong V\oplus I_{\mmm}$ as ad\,$\mmm$-modules.
It follows that
\begin{equation*}
B:=\{u\in U_\eta(\mmm)\mid I_{\mmm}u\subseteq I_{\mmm}\}=\{u\in U_\eta(\mmm)\mid[\mmm,u]\subseteq I_{\mmm}\}=V^{\mmm}\oplus I_{\mmm}.
\end{equation*}
This gives $B/I_{\mmm}\cong(Q^\eta_\chi)^{\text{ad}\,{\mmm}}\cong U_\eta({\ggg})^{\text{ad}\,{\mmm}}/ U_\eta({\ggg})^{\text{ad}\,{\mmm}}\cap I_{\mmm}$. Since $Q_\chi^\eta\cong U_\eta({\ggg})/I_{\mmm}$ as $U_\eta({\ggg})$-modules
we have $\text{End}_\ggg Q_\chi^\eta\cong(B/I_{\mmm})^{\text{op}}$ as algebras (see, e.g., \cite[Exercise 2.1.2(c)]{Pi}), which gives the desired result.
\end{proof}

Therefore, any left $U_{\eta}({\ggg},e)$-module can be considered as a $U_\eta({\ggg})^{\text{ad}\,{\mmm}}$-module with the trivial action of the ideal $U_\eta({\ggg})^{\text{ad}\,{\mmm}}\cap I_{{\mmm}}$. It follows from \cite[Theorem 1.3]{Skry} that every finite-dimensional $U_\eta(\ggg)$-module is $U_\eta(\mmm)$-free, and note that $U_\eta({\ggg})\cong\text{Mat}_{p^{\dim  \mathfrak{m}}}(U_{\eta}({\ggg},e))$ as $\bbk$-algebras by \cite[Lemma 2.2 (iii)]{Pre3}. Combining these with the fact that the restriction of $\eta$ coincides with that of $\chi$ on ${\mmm}$, by the same discussion as in \cite[Theorem 2.4]{Pre1} (the original result there is for the case with $\eta=\chi$)  we can conclude that
\begin{theorem}\label{Morita}
The functors
$$\vartheta:~U_\eta({\ggg})\text{-mod}\longrightarrow U_{\eta}({\ggg},e)\text{-mod},\qquad M\mapsto M^{{\mmm}},$$
and
$$\theta:~U_{\eta}({\ggg},e)\text{-mod}\longrightarrow U_\eta({\ggg})\text{-mod},\qquad V\mapsto U_\eta({\ggg})\otimes_{U_\eta({\ggg})^
{\text{ad}\,{\mmm}}}V,$$
are mutually inverse category equivalences.
\end{theorem}
\noindent Moreover, from the detailed proof of \cite[Theorem 2.4]{Pre1} we see that
\begin{equation}\label{MMm}
\dim M=p^{\dim  \mathfrak{m}}\cdot\dim M^{{\mmm}}=p^{\frac{\dim G.e}{2}}\cdot\dim M^{{\mmm}}.
\end{equation}

As an immediate corollary of Theorem \ref{Morita}, we have
\begin{prop}\label{lessthan}  All irreducible modules of $\ct(\ggg,e)$ are finite-dimensional with dimension at most $p^{\frac{\dim \ggg-\text{rank}\,\ggg-\dim G.e}{2}}$.
\end{prop}
\begin{proof}
Recall that all irreducible modules over $\ggg$ are finite-dimensional with maximal dimension equaling  $M(\ggg)=p^{\frac{\dim \ggg-\text{rank}\,\ggg}{2}}$
(see \cite[Theorem 3]{Mi} or \cite[Theorem 4.4]{PS}), so the dimension of any irreducible $U_{\eta}(\ggg)$-module with $\eta\in\chi+\mathfrak{m}^\perp$ is at most $M(\ggg)$.
Since $U_\eta(\ggg,e)\cong\ct_\eta(\ggg,e)$ as $\mathds{k}$-algebras by \S\ref{251},
it follows from Theorem \ref{Morita} and \eqref{MMm} that the dimension of any irreducible $\ct_{\eta}(\ggg,e)$-module over $\mathds{k}$ is at most $p^{\frac{\dim \ggg-\text{rank}\,\ggg-\dim G.e}{2}}$.

On the other hand, since we have $\ct(\ggg,e)\cong Q_\chi^{\mathcal{M}}$ by \eqref{equiW}, and $\ct_\eta(\ggg,e)=\ct({\ggg},e)/\ct({\ggg},e)K_\eta$ by Definition \ref{redu},
then any irreducible $\ct(\ggg,e)$-module is a simple object of the $\ct_\eta(\ggg,e)$-module category for some $\eta\in\chi+\mathfrak{m}^\perp$. Now the proof is completed.
\end{proof}

\section{The fractional ring of $\ct(\ggg,e)$ over the fractional field of its center}\label{fractional of}
In this section we continue to investigate the ring theoretic  properties on the centers of the finite $W$-algebra $\ct(\ggg,e)$, and also the ones on the extended finite $W$-algebra $U(\ggg,e)$.

Maintain the notations and assumptions as before. In particular, we list some special important conventions as below.

$\bullet$ The finite $W$-algebra $\ct(\ggg,e)$ will be simply written as  $\ct$ sometimes. In particular, this simplified notation will always appear in the context related to the centers as below.

$\bullet$ The $p$-center $Z_0(\tilde{\mathfrak{p}})^{\mathcal{M}}$ of $\ct$ (see Definition \ref{pcenterW}) will be written as  $Z_0(\ct)$ for convenience.

$\bullet$ The center of $\ct$ will be denoted by  $Z(\ct)$ (Keep in mind the fact that  $Z(\ct)$ is  an integral domain by Lemma \ref{Noe Pri}(3)).

$\bullet$ Set $\bbf_0:=\text{Frac}(Z_0(\ct))$ and $\bbf:=\Frac(Z(\ct))$ the fractional field of $Z_0(\ct)$ and $Z(\ct)$, respectively.

$\bullet$ Set the fractional ring $$Q(\ct):=\ct(\ggg,e)\otimes_{Z(\ct)}\bbf.$$

$\bullet$ Write $\ell:={n-r-\dim G.e}$, where $n=\dim\ggg$ and $r=\rank\,\ggg$.

\subsection{$Z(\ct)$ and $Z_0(\ct)$}\label{centers of ct and its subalgebras}
We first have the following observation:
\begin{prop}\label{PI and Central S}
The following are true:
\begin{itemize}
\item[(1)] the ring $Z_0(\ct)$ is  Noetherian;
\item[(2)] the ring $Z(\ct)$ is integrally closed;
\item[(3)] the ring extension $Z(\ct)/Z_0(\ct)$ is integral, and $Z(\ct)$ is finitely-generated as a $Z_0(\ct)$-module. In particular, $Z(\ct)$ is a finitely-generated commutative algebra over $\bbk$;
\item[(4)]$\ct(\ggg,e)$ is a PI ring.
\end{itemize}
\end{prop}

\begin{proof}
(1) By Lemma \ref{pmge}(1) we know that
$Z_0(\ct)\cong Z_0(\ggg_e)$ as $\mathds{k}$-algebras, then $Z_0(\ct)$ is isomorphic to a polynomial algebra in $\dim \ggg_e$ variables,
which is a Noetherian ring.

(2) Generally speaking, the proof can be carried out by exploiting \cite[\S6, Propositions 5.3-5.4]{SF} which  contributes to the centers of Lie algebras. We give some details below for the convenience of the readers.

(Step 1)
Let $x$ be contained in the integral closure of $Z(\ct)$ in $\bbf$. Then the element $x$ satisfies an equation
$$x^m+k_{m-1}x^{m-1}+\cdots+k_0=0,$$where $k_i\in Z(\ct)$ for $0\leqslant i\leqslant m-1$. Set $\mathbf{B}:=\sum_{i=0}^{m-1}\ct(\ggg,e)x^i$. Note that $x$ is in the center of $Q(\ct)$. So $\mathbf{B}$ is a sub-ring of the fraction ring $Q(\ct)$ of $\ct(\ggg,e)$. Write $x=\frac{r}{s}$ with $r\in Z(\ct)$ and $s\in Z(\ct)\backslash\{0\}$. Then $\ct(\ggg,e)\subseteq \mathbf{B}\subseteq\frac{1}{z}\ct(\ggg,e)$, where $z:=s^m\in Z(\ct)$.

(Step 2) Now we  prove $\mathbf{B}=\ct(\ggg,e)$.
Recall that the Kazhdan filtration on $\ct({\ggg},e)$ is defined by
$$\text{F}_0\ct({\ggg},e)\subseteq \text{F}_{1}\ct({\ggg},e)\subseteq\cdots$$ with $\text{F}_0\ct({\ggg},e)=\mathds{k}$. We claim that \begin{equation}\label{zbu}
(z\mathbf{B})\cap\text{F}_k\ct({\ggg},e)=(z\ct(\ggg,e))\cap\text{F}_k\ct({\ggg},e)
\end{equation} for $k\geqslant0$ and prove it by induction on $k$.

As $\text{F}_0\ct({\ggg},e)=\mathds{k}$, it is obvious that \eqref{zbu} holds for the case with $k=0$.
Now let $a\in(z\mathbf{B})\cap\text{F}_{k+1}\ct({\ggg},e)$. For $n\geqslant1$, we have $a^n\in z^n\mathbf{B}\subseteq z^{n-1}\ct(\ggg,e)$. Then there exists $u_n\in\ct(\ggg,e)$ such that $a^n=z^{n-1}u_n$. Under the Kazhdan grading we can choose $\bar a, \bar z, \bar u_n$ in $\text{gr}(\ct(\ggg,e))$ such that $\text{gr}(a)=\bar a, \text{gr}(z)=\bar z$ and $\text{gr}(u_n)=\bar u_n$, respectively.
Recall in Lemma \ref{Noe Pri}(1) we showed that $\text{gr}(\ct(\ggg,e))$ is a unique factorization domain.  Let $q$ be a prime factor of $\bar z$. Then there are $l, t\in\mathbb{Z}_+$ such that $q^l\mid \bar z$, $q^{l+1} \nmid \bar z$, $q^t\mid \bar a$, $q^{t+1} \nmid \bar a$.
As $\bar z^{n-1}\mid\bar a^n$, we have  $(n-1)l\leqslant nt$ for all $n$, which yields $t\geqslant l$. This shows that $\bar z$ divides $\bar a$. Therefore, there exists $v\in\ct(\ggg,e)$ such that $\bar a=\bar z\bar v$, whence $a-zv\in\text{F}_k\ct(\ggg,e)$. Since $\ct(\ggg,e)\subseteq \mathbf{B}$ and $a\in z\mathbf{B}$, we have $a-zv\in\text{F}_k\ct(\ggg,e)\cap(z\mathbf{B})$. The induction hypothesis implies that $a-zv\in\text{F}_k\ct(\ggg,e)\cap(z\ct(\ggg,e))$. Hence $a\in\text{F}_{k+1}\ct(\ggg,e)\cap(z\ct(\ggg,e))$, as desired.

(Step 3) It follows from \eqref{zbu} that $z\ct(\ggg,e)=z\mathbf{B}$. By Lemma \ref{Noe Pri}(3), $\ct(\ggg,e)$ has no zero divisors, and we  obtain $x\in \mathbf{B}\subseteq\ct(\ggg,e)$. It follows that $\mathbf{B}=\ct(\ggg,e)$.

(3) As $Z_0(\ct)$ is a Noetherian ring by Statement (1) and  $\ct(\ggg,e)$ is a finitely-generated $Z_0(\ct)$-module by  Theorem \ref{prem}(4),  $\ct(\ggg,e)$ is a Noetherian $Z_0(\ct)$-module. Since $Z(\ct)$ is a $Z_0(\ct)$-submodule in $\ct(\ggg,e)$,  $Z(\ct)$ is also finitely-generated as a $Z_0(\ct)$-module.
Then by a standard argument as in \cite[Proposition 2.4]{AM}, one can easily conclude that $Z(\ct)$ is integral over $Z_0(\ct)$.
In  (1) we showed that $Z_0(\ct)$
is isomorphic to a polynomial algebra in $\dim\ggg_e$ variables, and also $Z(\ct)$ is finitely-generated as a $Z_0(\ct)$-module, then $Z(\ct)$ is a finitely-generated commutative algebra over $\bbk$.

(4) Since  $\ct(\ggg,e)$ is a free  module over $Z_0(\ct)$ by Theorem \ref{prem}(4), \cite[Corollary 13.1.13(iii)]{MR} implies that $\ct(\ggg,e)$ is a PI ring.
\end{proof}

\subsection{On the structure of $Q(\ct)$}
We first have a theory for $Q(\ct)$  parallel to the modular Lie algebra case (cf. \cite{Za}).
\begin{prop}\label{centersimple}
The following statements hold.
\begin{itemize}
\item[(1)] $Q(\ct)\cong\ct(\ggg,e)\otimes_{Z_0(\ct)}\bbf_0$;
\item[(2)] $Q(\ct)$ is a division algebra;
\item[(3)] $Q(\ct)$ is a finite-dimensional and central simple algebra over $\bbf$;
\item[(4)] the dimension of $Q(\ct)$ over $\bbf$ is equal to $p^{2m}$ for some $m\in \bbz_+$ such that $2m\geqslant \ell$. Here as before  $\ell={n-r-\dim G.e}$ with $n=\dim\ggg$ and $r=\rank\,\ggg$.
\end{itemize}
\end{prop}
\begin{proof} The arguments are the same in the spirit as in \cite{Za} for modular Lie algebras. We represent them for modular finite $W$-algebra.

(1) Recall  that $Z(\ct)/Z_0(\ct)$ is an integral extension by Proposition \ref{PI and Central S}(3). Let $x\in Z(\ct)$ satisfy an equation
$$x^m+k_{m-1}x^{m-1}+\cdots+k_0=0,$$
where $k_i\in Z_0(\ct)$ for $0\leqslant i\leqslant m-1$. As Lemma \ref{Noe Pri}(3) entails that $\ct(\ggg,e)$ has no zero divisors and $Z(\ct)\subseteq\ct(\ggg,e)$, we may assume that $k_0\neq0$. Then $x$ is invertible in $(Z_0(\ct)\backslash\{0\})^{-1}Z(\ct)$ with inverse $x^{-1}=-k_0^{-1}(x^{m-1}+k_{m-1}x^{m-2}+\cdots+k_{1})$. Since any element in $Q(\ct)$ can be written as $\frac{t}{s}$ with $t\in\ct(\ggg,e)$ and $s\in Z(\ct)\backslash\{0\}$,  Statement (1) follows.

(2) Note that $\ct(\ggg,e)$ is finitely-generated over $Z_0(\ct)$ by  Theorem \ref{prem}(4), and $Z_0(\ct)\subseteq Z(\ct)$ by definition. Thus, $\ct(\ggg,e)$ is also finitely-generated over $Z(\ct)$. By the same discussion as in Proposition \ref{PI and Central S}(3), one can conclude that $\ct(\ggg,e)/Z(\ct)$ is an integral extension. Applying the considerations noted in the proof for Statement (1), one   concludes that $Q(\ct)$ is a division algebra.

(3) As $\ct(\ggg,e)$ is finitely-generated over  
$Z(\ct)$, it is straightforward that $Q(\ct)$ is finite-dimensional over $\bbf$. In virtue of Lemma \ref{Noe Pri}(3) and Proposition \ref{PI and Central S}(4), Posner's Theorem \cite[Theorem 13.6.5]{MR} shows that $Q(\ct)$ is a central simple algebra over $\bbf$.

(4) Recall that $\ct(\ggg,e)$ is a free module over $Z_0(\ct)$ of rank $p^{\dim\ggg_e}$ by Theorem \ref{prem}(4). Consequently, $Q(\ct)$ is an $\bbf_0$-vector space of dimension $p^{\dim\ggg_e}$ and we have  $$p^{\dim\ggg_e}=[Q(\ct):\bbf_0]=[\bbf:\bbf_0]\cdot[Q(\ct):\bbf].$$
This shows that $\dim_{\bbf} Q(\ct)$ is a power of $p$.
Now we already know that $Q(\ct)$ is a central simple algebra over $\bbf$, thus
this power is an even number.
Set $\dim_{\bbf} Q(\ct)=p^{2m}$ for some $m\in\mathbb{Z}_+$. In the following we will show that $2m\geqslant\ell$, and  the proof is divided into a couple of steps.

(Step 1) For $\eta\in \chi+\mmm^\perp$, there exists $g\in G$ such that $(g.\eta)(\mathfrak{n}^+)=0$ (see \cite[Lemma 6.6]{Jan2}). Set $\eta'=g.\eta$ with Jordan-Chevalley decomposition $\eta'=\eta'_s+\eta'_n$. Then $\eta'_s(\nnn^-)=0$, $\eta'_n(\hhh)=0$ and  the centralizer $\ggg_{\eta'_s}$ contains $\hhh$. There is a projection
$\pi: \hhh^*\rightarrow ([\ggg_{\eta'_s},\ggg_{\eta'_s}])\cap \hhh)^*$.
Consider the baby Verma module $V_\eta(\lambda)$ in $U_{g.\eta}(\ggg)$-modules  defined to be
$$V_\eta(\lambda)=U_{g.\eta}(\ggg)\otimes_{U_{g.\eta}(\bbb)}\bbk_{\lambda},$$
where $\bbb$ is a Borel subalgebra of $\ggg$ associated to the root system $\Phi$, and $\lambda\in \pi^{-1}(\pi(-\rho))$ satisfies that $\forall H\in\hhh$, $\lambda(H)^p-\lambda(H^{[p]})=\eta_s'(H)^p$ with $\rho$ being half sum of the positive roots $\Phi^+$ of $\ggg$.
According to  \cite[Proposition 3.14]{BGo}, $V_\eta(\lambda)$ is an irreducible  $U_{g.\eta}(\ggg)$-module and whose dimension over $\mathds{k}$ equals $p^{{1\over 2}(\dim\ggg-\rank\,\ggg)}$. Since $U_{\eta}(\ggg)\cong U_{g.\eta}(\ggg)$ as $\mathds{k}$-algebras, $V_\eta(\lambda)$ can also be considered as a $U_{\eta}(\ggg)$-module.
 By Theorem \ref{Morita} there are mutually inverse category equivalences from $U_\eta(\ggg)$-modules     to $U_\eta(\ggg,e)$-modules.
As the discussion in \S\ref{251} shows that $U_\eta({\ggg},e)\cong\ct_\eta({\ggg},e)$ as $\mathds{k}$-algebras (we shall identify $\ct_\eta({\ggg},e)$ with $U_\eta({\ggg},e)$ below),  it follows from \eqref{MMm} that $\vartheta(V_\eta(\lambda))$ is an irreducible $\ct_\eta(\ggg,e)$-module  with $\dim\vartheta(V_\eta(\lambda))=p^\frac{\ell}{2}$.

Let us introduce a canonical algebra homomorphism $\Theta$ from $\ct(\ggg,e)$ to
$\ct(\ggg,e)\slash(z-\Theta_{(\eta,\lambda)}(z))$, where
$(z-\Theta_{(\eta,\lambda)}(z))$ is the ideal of $\ct(\ggg,e)$
generated by $z-\Theta_{(\eta,\lambda)}(z)$ with $z$ taking all through the center $Z(\ct)$ of $\ct(\ggg,e)$, and
$\Theta_{(\eta,\lambda)}(z)\in\bbk$ is defined by $z.v=\Theta_{(\eta,\lambda)}(z)v$ for $v\in \vartheta(V_\eta(\lambda))$.
Note that the category of
$\Theta(\ct(\ggg,e))$-modules is a subcategory of
$\ct_{\eta}(\ggg,e)$-modules (see Remark \ref{Central Character} below), and $\ct_\eta({\ggg},e)$ is the reduced algebra of $\ct({\ggg},e)$ by definition. Then we can regard $\vartheta(V_\eta(\lambda))$ as an irreducible $\ct(\ggg,e)$-module, which will be denoted by $V_\ct$. Set $\varrho$ to be the corresponding map. By Jacobson's density theorem, $
\varrho(\ct(\ggg,e))=\End_\bbk(V_\ct)
$ (see \cite[\S4.3]{Jaco}).

(Step 2)
In virtue of (2) and (3), $Q(\ct)$ is a finite-dimensional division $\bbf$-algebra. Then for any $u\in \ct(\ggg,e)$ we have a subfield $\bbf(u)$ of $Q(\ct)$, which actually coincides with $\bbf[u]$. By general theory, $\bbf[u]$ has degree $\leqslant p^m$ over $\bbf$ (see e.g., \cite[Theorem 4.5.1]{DV}). According to Proposition \ref{PI and Central S}(2),
$Z(\ct)$ is integrally closed.  An application of \cite[\S6, Lemma 5.2(2)]{SF} to $\bbf[u]$ entails that the minimal polynomial $f_u$ of $u$ in $\bbf[\tau]$ lies in $Z(\ct)[\tau]$. Since $V_\ct$ is irreducible and $\mathds{k}$ is algebraically closed, we have $\varrho(Z(\ct))=\mathds{k}\,\text{id}_{V_\ct}$. Hence $\varrho(u)$ satisfies a polynomial of standard degree $\text{deg}_S(f_u)\leqslant p^m$ over $\mathds{k}$.




(Step 3) We claim that $\dim V_\ct\leqslant p^m$.
In (Step 1) we showed that
 $\varrho(\ct(\ggg,e))=\End_\bbk(V_\ct)$.
So there must be  an element $u\in \ct(\ggg,e)$ such that $\varrho(u)\in\End_\bbk(V_\ct)$ is equal to the standard regular nilpotent matrix
\begin{equation*}
\left( \begin{array}{ccccc}
 0 & 1 & 0 &\cdots & 0 \cr
0 &0 &1 &\cdots &0 \cr
0&0&0&\cdots &0 \cr
\vdots&\vdots&\vdots&\vdots&\vdots\cr
0&0&0&\cdots &1\cr
0&0&0&\cdots &0
\end{array}\right).
\end{equation*}
So the degree of minimal polynomial of $\varrho(u)$ over $\bbk$ is $\dim V_\ct=p^{\frac{\ell}{2}}$. Using (Step 2), we finally obtain that $p^{\frac{\ell}{2}}\leqslant p^m$. Thus we have  $2m\geqslant\ell$, completing the proof of Statement (4).
\end{proof}
\begin{remark}\label{Central Character} Recall that for any irreducible $U(\ggg)$-module $S$ with $p$-character $\eta$, there is a central character $\lambda_\eta(S):Z(\ggg)\rightarrow \bbk$. Denote by $\lambda_\eta^0(S)$ the corresponding $p$-central character which is equal to $\lambda_\eta(S)|_{Z_0(\ggg)}$. 
Then $\ker(\lambda_\eta(S))$ and $\ker(\lambda_\eta^0(S))$ are the maximal ideals of $Z(\ggg)$ and $Z_0(\ggg)$ respectively. %
The same thing happens on $\ct(\ggg,e)$.
\end{remark}
\begin{remark}
Set $Z(U(\ggg,e))$ to be the center of $U(\ggg,e)$. In virtue of Lemma$\;$\ref{Noe Pri}(3) and Theorem \ref{prem}(2), $Z(U(\ggg,e))\subseteq U(\ggg,e)$ must be an integral domain. Let $\text{Frac}(Z_0(\tilde{\mathfrak{p}}))$ and $\Frac(Z(U(\ggg,e)))$ denote the fractional field of $Z_0(\tilde{\mathfrak{p}})$ and $Z(U(\ggg,e))$ respectively. Parallel to the situation of $\ct(\ggg,e)$, we have the following fractional ring $$Q(U(\ggg,e)):=U(\ggg,e)\otimes_{Z(U(\ggg,e))}\Frac(Z(U(\ggg,e))).$$
Then the same statements as in Propositions \ref{PI and Central S}, \ref{centersimple} and \ref{lessthan} hold for the extended finite $W$-algebra ${U}(\ggg,e)$, i.e.,
\begin{itemize}
\item[(1)] $Z_0(\tilde{\mathfrak{p}})$ is a  Noetherian ring;
\item[(2)] the ring $Z(U(\ggg,e))$ is integrally closed;
\item[(3)] $Z(U(\ggg,e))$ is integral over $\mathds{k}$-algebra $Z_0(\tilde{\mathfrak{p}})$, and $Z(U(\ggg,e))$ is finitely-generated as a $Z_0(\tilde{\mathfrak{p}})$-module. In particular, $Z(U(\ggg,e))$ is a finitely-generated commutative algebra over $\bbk$;
\item[(4)] $U(\ggg,e)$ is a PI ring;
\item[(5)] $Q(U(\ggg,e))\cong U(\ggg,e)\otimes_{Z_0(\tilde{\mathfrak{p}})}\text{Frac}(Z_0(\tilde{\mathfrak{p}}))$;
\item[(6)] $Q(U(\ggg,e))$ is a division algebra;
\item[(7)] $Q(U(\ggg,e))$ is a finite-dimensional central simple algebra over $\Frac(Z(U(\ggg,e)))$ with dimension not less than $p^\ell$, where $\ell={\dim\ggg-\rank\,\ggg-\dim  G.e}$;
\item[(8)] all irreducible modules of $U(\ggg,e)$ are finite-dimensional with dimension less than $p^{\frac{\ell}{2}}$.
\end{itemize}
\end{remark}

\section{Centers of finite $W$-algebras: preliminary lemmas}\label{On the centers of}
Keep the notations and assumptions as before.
 In particular, we have a map $\varphi:~Z(\ggg)\rightarrow U(\ggg,e)$ (see \S\ref{p-center}),  which will play a key role in the sequent.
\subsection{Background}\label{bac}
Let us first observe the centers of finite $W$-algebras over $\bbc$. According to Premet's arguments in the footnote of \cite[Question 5.1]{P3}, there is an isomorphism between $\bbc$-algebras:
$$\varphi_\bbc:~Z(\ggg_\bbc)\cong Z(U(\ggg_\bbc,\hat e)),$$
where $Z(\ggg_\bbc)$ and $Z(U(\ggg_\bbc,\hat e))$ denote the center of the universal enveloping algebra $U(\ggg_\bbc)$ and the center of the finite $W$-algebra $U(\ggg_\bbc,\hat e)$ respectively, and the map $\varphi_\bbc$ is the counterpart of  $\varphi$.

Recall that under some mild assumption on $p$,
 $Z(\ggg)$ is  generated by two parts: the Harish-Chandra center $Z_1(\ggg):=U(\ggg)^G$ of the adjoint $G$-action,  and the $p$-center $Z_0(\ggg)$. 
Now we turn back to the finite $W$-algebra case. Let $Z(\ct)$ be the center of the finite $W$-algebra $\ct(\ggg,e)$. Recall that the $p$-center of $\ct(\ggg,e)$ is defined as $Z_0(\ct)=Z_0(\tilde{\mathfrak{p}})^{\mathcal{M}}$ in Definition \ref{pcenterW}. On the other hand,
by the definition of $\ct(\ggg,e)$ we know that the image of $Z_1(\ggg)$ under the map $\varphi$ in \S\ref{p-center} falls into $\ct(\ggg,e)$, which is denoted by $Z_1(\ct)$.
By definition, $Z_0(\ct)$ and $Z_1(\ct)$ lie in the center of $Z(\ct)$. A natural question rises:
\begin{question}\label{question}
Is $Z(\ct)$ generated by the $\bbk$-algebras $Z_0(\ct)$ and $Z_1(\ct)$?
\end{question}

First we recall some basics on the Harish-Chandra homomorphism in positive characteristic (one refers to \cite[\S9.1]{Jan2} for more details). For the triangular decomposition ${\ggg}=\mathfrak{n}^-\oplus\mathfrak{h}\oplus\mathfrak{n}^+$ of $\ggg$, we can define a linear map
\begin{equation}\label{pi}
\pi:U(\ggg)\rightarrow U(\mathfrak{h})
\end{equation} as the projection with kernel $\mathfrak{n}^-U(\ggg)+U(\ggg)\mathfrak{n}^+$. Clearly $U(\ggg)^{G}$ is contained in the $0$ weight space of $U(\ggg)$ with the adjoint action of the maximal torus $T$, so $\pi$ restricts to an algebra homomorphism $\HC:~U(\ggg)^{G}\rightarrow U(\mathfrak{h})$. Thanks to \cite[Lemma 9.1]{Jan2}, $\HC$ is injective.

Now we will deal with Question \ref{question}.  
The following observation is fundamental. 
\begin{lemma}\label{Har Chand Cent}
There exists an isomorphism between $\mathds{k}$-algebras
$$\varphi|_{U(\ggg)^{G}}:~U(\ggg)^G\xrightarrow\sim Z_1(\ct),$$
where $\varphi|_{U(\ggg)^{G}}$ denotes the restriction of the map $\varphi$ in \S\ref{p-center} to the subalgebra $U(\ggg)^{G}$ of $Z(\ggg)$.
\end{lemma}
\begin{proof} It follows from the same arguments as in
\cite[\S 6.2]{Pre1} for the case over $\bbc$.
Actually, by \eqref{equiW}  we identify $\ct(\ggg,e)$ with $Q_\chi^{\mathcal{M}}$. For any nonzero $z\in U(\ggg)^G\subseteq Z(\ggg)$, apply $z$ to the canonical generator $1_\chi$ of $Q_\chi$ and then express $z(1_\chi)\in Q_\chi^{\mathcal{M}}$ via its monomial basis. Note that the Harish-Chandra projection  $\HC:~U(\ggg)^{G}\rightarrow U(\mathfrak{h})$ is injective, and $z$ has the same standard filtration degree as its projection $\HC(z)$ in the polynomial ring $U(\hhh)$. Hence that $z(1_\chi)\ne 0$ because the terms of $z$  other than the projection $\HC(z)\ne0$ will
 not be able to cancel $\HC(z)(1_\chi)$ which implies that $z(1_\chi)\ne 0$.

To complete the proof, it remains to note that the image of $U(\ggg)^{G}$ under the map $\varphi|_{U(\ggg)^{G}}$ is just $Z_1(\ct)$.
%
%
%
\end{proof}
\subsection{Some lemmas}\label{3.2}
Let $W$ be the (abstract) Weyl group of $G$ associated with a given Cartan subalgebra $\mathfrak{h}$. Define the dot action of $W$ on $\hhh^*$ via 
$w\centerdot\lambda=w(\lambda+\rho)-\rho$ for  $\lambda\in\mathfrak{h}^*$ and $w\in W$, where $\rho$  is half-sum of all of the positive roots $\Phi^+$ of $\ggg$. The dot action on $\mathfrak{h}^*$ yields also a dot action on $U(\mathfrak{h})$. As $\mathfrak{h}$ is commutative, we can identify $U(\mathfrak{h})$ with the symmetric algebra $S(\mathfrak{h})$, hence with the algebra of polynomial functions of $\mathfrak{h}^*$. Considering $f\in U(\mathfrak{h})$ as a function on $\mathfrak{h}^*$  we define $w\centerdot f$ for $w\in W$ by $(w\centerdot f)(\lambda)=f(w^{-1}\centerdot\lambda)$. Then Harish-Chandra's theorem in \cite[Proposition 2.1]{Ve}  shows that $\HC:~U(\ggg)^G\cong U(\hhh)^{W\centerdot}$ as $\bbk$-algebras, where $U(\hhh)^{W\centerdot}$ is  the invariant subring of $U(\hhh)$ under the action of $W\centerdot$. Moreover, \cite[\S2]{Ve} entails that $U(\hhh)^{W\centerdot}$ is isomorphic to a polynomial algebra $\bbk[T_1,\cdots,T_r]$ for $r=\rank\,\ggg$. We choose  $\{g_1,\cdots,g_r\}$ as a set of algebraically independent generators in $U(\ggg)^G$ such that $\HC(g_i)=T_i$ for $1\leqslant i\leqslant r$. We can further assume that $\text{deg}_Sg_i=m_i+1$ for $1\leqslant i\leqslant r$ associated to the standard grading of $U(\ggg)$, where $m_1,\cdots,m_r$ are the exponents of the Weyl group of $\ggg$. Set $f_i:=\varphi(g_i)$ with $1\leqslant i\leqslant r$. Then the $f_i$'s are elements in $Z_1(\ct)$ by definition.

Recall the notation $\Lambda_k=\{(i_1,\cdots,i_k)\mid i_j\in\{0,1,\cdots,p-1\}\}$
for $k\in\mathbb{Z}_+$ with $1\leqslant j\leqslant k$. Then we have

\begin{lemma}\label{Veldkamp mod}
The elements $f_1^{t_1}\cdots f_r^{t_r}$ with $(t_1,t_2,\cdots,t_r)$ running through $\Lambda_r$ generate a free module $M$ over $Z_0(\ct)$. Moreover, these elements form a $Z_0(\ct)$-basis for $M$.
\end{lemma}

\begin{proof} Note that we already  have $Z_0(\tilde{\mathfrak{p}})\cap\ker(\varphi)=\{0\}$, and  $\varphi(Z_0(\ggg))\cong\varphi(Z_0(\tilde{\mathfrak{p}}))\cong Z_0(\tilde{\mathfrak{p}})$ by \S\ref{pcen}.
 As $Z_0(\ct)\subseteq Z(\ct)\subseteq\ct(\ggg,e)$,  it follows from \eqref{twoequation} that
$Z_0(\ct)=Z(\ct)\cap\varphi(Z_0(\tilde{\mathfrak{p}}))$. Especially, we have $Z_0(\ct)\subseteq \varphi (Z_0(\tilde{\mathfrak{p}}))$. Combining this with Lemma \ref{Har Chand Cent}, we can obtain
\begin{align}\label{Augmented ideal}
\ker(\varphi|_{Z(\ggg)})\subseteq Z_0(\mmm_\chi)^+Z(\ggg),
\end{align}
where $Z_0(\mmm_\chi)^+$ is the ideal of $Z_0(\mmm)$ generated by $x^p-x^{[p]}-\chi(x)^p$ for $x\in\mmm$. Given $\mathbf{t}:=(t_1,t_2,\cdots,t_r)\in\Lambda_r$,
let $\mathbf{g}^{\mathbf{t}}:=g_1^{t_1}\cdots g_r^{t_r}$ be the monomial in $U(\ggg)^G$. Denote by $N:=\{\mathbf{g}^{\mathbf{t}}\mid\bt\in\Lambda_r\}$, and let
$\ca$ be the $Z_0(\tilde{\mathfrak{p}})$-module generated by all the elements in $N$. By Veldkamp's theorem (\cite[Theorem 3.1]{Ve}) we can define a map $\varphi|_\ca: \ca\rightarrow \varphi(\ca)\subseteq Z(U(\ggg,e))$.

For $\mathbf{t}=(t_1,t_2,\cdots,t_r)\in\Lambda_r$,
denote $\mathbf{f}^{\mathbf{t}}:=f_1^{t_1}\cdots f_r^{t_r}$ which is equal to $\varphi(\mathbf{g}^{\mathbf{t}})$, a monomial in $Z_1(\ct)$.
To complete the proof of the lemma, we only need to prove that for any given equation
\begin{equation}\label{at}
\sum_{\mathbf{t}\in \Lambda_r} a_{\mathbf{t}}\mathbf{f}^{\mathbf{t}}=0 \mbox{ where the coefficients }  a_{\mathbf{t}}\in Z_0(\ct),
\end{equation}
these $a_\bt$ 
must be zero. As $Z_0(\ct)$ is a subalgebra of $\varphi(Z_0(\tilde{\mathfrak{p}}))$, we can write   $a_{\bt}=\varphi(b_\bt)$ for some $b_\bt\in Z_0(\tilde{\mathfrak{p}})$ when $\mathbf{t}$ ranges in $\Lambda_r$. Then
 $$\sum_{\mathbf{t}\in \Lambda_r} a_{\mathbf{t}}\mathbf{f}^{\mathbf{t}}=\sum_{\mathbf{t}\in \Lambda_r} \varphi (b_{\mathbf{t}})\varphi(\mathbf{g}^{\mathbf{t}})=\sum_{\mathbf{t}\in \Lambda_r} \varphi (b_{\mathbf{t}}\mathbf{g}^{\mathbf{t}})\in\varphi(\ca).$$
 Due to the injectivity of the map $\varphi|_\ca$ (see Lemma \ref{forthcoming} below),  the following equation holds in $Z(\ggg)$:
 $$\sum_{\mathbf{t}\in \Lambda_r} b_{\mathbf{t}}\mathbf{g}^{\mathbf{t}}=0.$$
 By Veldkamp's theorem, all the coefficients $b_{\mathbf{t}}$ with $\mathbf{t}\in \Lambda_r$ must be zero, thereby all the coefficients $a_\bt$ in \eqref{at} must be zero. The proof is completed (modulo Lemma \ref{forthcoming}).
\end{proof}

Before introducing Lemma \ref{forthcoming}, we make some preparation.  For any  $\eta\in \chi+\mmm^\perp\subseteq \ggg^*$, we  have a canonical projection
$$ {\Pr}_{\eta}: U(\ggg,e)\rightarrow U^\circ_\eta({\ggg},e)=U(\ggg,e)/H_\eta$$
by \eqref{quo}.
Set $\theta_\eta:={\Pr}_{\eta}\circ \varphi: Z(\ggg)\rightarrow U^\circ_\eta({\ggg},e)$.

\begin{lemma}\label{forthcoming} Let $\ca$ be as in the proof of Lemma \ref{Veldkamp mod}. Then $\ca\cap \ker(\varphi)=\{0\}$.
\end{lemma}

\begin{proof} Keep the notations as above.
We  prove the lemma by reductio ad absurdum. Suppose $\ca\cap\ker(\varphi)\ne\{0\}$. Then we have
\begin{equation}\label{spannoeta}
\sum_{\bt\in\Lambda_r} a_\bt\df^\bt=0
\end{equation}
for some $a_\bt=\varphi(b_\bt)\in \varphi(Z_0(\tilde{\mathfrak{p}}))$ which are  not all zero. We proceed the arguments in steps.

(i)  We claim that for any non-zero $b_\bt\in Z_0(\tilde{\mathfrak{p}})$ with $\bt\in\Lambda_r$, there exists some $\eta\in \chi+\mmm^\perp$ such that $\theta_\eta(b_\bt)\ne 0$.
This follows from the fact that $Z_0(\tilde{\mathfrak{p}})$ equals the coordinate algebra of the Frobenius twist $(\chi+\mathfrak{m}^\perp)^{(1)}$ of $\chi+\mathfrak{m}^\perp$; see \S\ref{312} for more details.

(ii)  We claim that for any $\eta\in  \chi+\mmm^\perp$, $\theta_\eta(\bg^\bt)={\Pr}_{\eta}(\df^\bt)$ for all $\bt\in \Lambda_r$ are linearly independent in $U^\circ_\eta(\ggg,e)$.
In order to show this, we firstly observe that under the map $\theta_\eta$ the image of $Z_0(\tilde{\mathfrak{p}})$ falls into $\bbk$.

Recall that there is an isomorphism $\psi_\eta: U^\circ_\eta(\ggg,e)\rightarrow U_\eta(\ggg,e)=(\text{End}_\ggg Q_\chi^\eta)^{\text{op}}$ by \eqref{reduced W3}.
Denote by $Z^\eta(\ggg)$ the image of $Z(\ggg)$ under the canonical homomorphism
\begin{equation*}\label{Pi}
\tau_\eta:~~U({\ggg})\rightarrow U_\eta({\ggg}),
\end{equation*}
and define $\varphi_\eta$ to be a homomorphism of $\mathds{k}$-algebras from $Z^\eta(\ggg)$ to $U_\eta(\ggg,e)$ via
$x\mapsto l_x$,
where $l_x(1_\chi)=x.1_\chi\in Q_\chi^\eta$ for any $x\in Z^\eta(\ggg)$. Moreover, since $U_{\eta}({\ggg},e)\cong(Q^\eta_\chi)^{\text{ad}\,{\mmm}}$ by Lemma \ref{redum},
we can similarly define $\varphi_\eta(x)$ in $(Q^\eta_\chi)^{\text{ad}\,{\mmm}}$ and regard it as an element of $U_{\eta}(\ggg,e)$ in the following.

 Secondly,  the following diagram is commutative
\[
\begin{CD}
Z(\ggg) @>{\varphi}>> U(\ggg,e)@>{{\Pr}_{\eta}}>> U^\circ_\eta(\ggg,e)\\
@V{\tau_\eta}VV @V{\psi_\eta\circ{\Pr}_{\eta}}VV @VV{=}V \\
Z^\eta(\ggg) @>>{\varphi_\eta}>U_\eta(\ggg,e)@<<{\psi_\eta}<{U}^\circ_\eta(\ggg,e)
\end{CD}
\]
Actually, the commutativity is assured by the definitions of those algebras and the isomorphism $\psi_\eta$. We will identify $U_\eta(\ggg,e)$ with $U^\circ_\eta(\ggg,e)$ in the following.

Thirdly, by Veldkamp's theorem  $Z(\ggg)$ is a free $Z_0(\ggg)$-module with a basis consisting of all $\bg^\bt$ with $\bt\in\Lambda_r$. Therefore, the elements $\tau_\eta(\bg^\bt)$ are linearly independent in $Z^\eta(\ggg)$.

 Lastly, note that the elements $\bg^\bt\in U(\ggg)^G$ with $\bt\in\Lambda_r$  come from the Harish-Chandra homorphism $\HC: U(\ggg)^G\rightarrow U(\hhh)$, and we have $\eta|_\mmm=\chi|_\mmm$ by definition. Denote by $\mathsf{HC}^\eta_\chi$ the subspace of $U_\eta(\ggg)$ spanned by all $\tau_\eta(\bg^\bt)$ with $\bt\in\Lambda_r$. By the same discussion as in the proof of Lemma \ref{Har Chand Cent},
 the map $\varphi_\eta|_{\mathsf{HC}^\eta_\chi}$ is injective. The claim follows from the above commutative diagram.

(iii)  By (i),  for any non-zero $b_\bt\in Z_0(\tilde{\mathfrak{p}})$ with $\bt\in\Lambda_r$, we can always find some $\eta\in \chi+\mmm^\perp$ such that $\theta_\eta(b_\bt)\ne 0$.
From \eqref{spannoeta} we have
\begin{align}\label{nontrival span}
\sum_{\bt\in\Lambda_r} \theta_\eta(b_\bt)\theta_\eta(\bg^\bt)=0.
\end{align}
The summand on the left hand side of \eqref{nontrival span} lies in $U_\eta(\ggg,e)$.
Keep in mind  $\theta_\eta(b_\bt)\in\bbk$. The left hand side of (\ref{nontrival span}) is a non-trivial $\bbk$-span of $\{\theta_\eta(\bg^\bt)\mid \bt\in \Lambda_r\}$, contradicting the claim in (ii) that $\theta_\eta(\bg^\bt)$ for all $\bt\in \Lambda_r$ are linearly independent.

By summing up, we accomplish the arguments.
\end{proof}
\begin{rem}\label{remformain1}
In virtue of Lemma \ref{forthcoming}, by careful inspection we can see that an analogue of Lemma \ref{Veldkamp mod} given by substituting $Z_0(\tilde{\mathfrak{p}})$ for $Z_0(\ct)$ is still true.
\end{rem}
Denote by $\cz$ the free $Z_0(\ct)$-module generated by all $\{\mathbf{f}^\bt\mid \bt\in \Lambda_r\}$ as introduced in Lemma \ref{Veldkamp mod}.  And denote by $\tilde\cz$ the subalgebra of $Z(\ct)$ generated by $Z_0(\ct)$ and $Z_1(\ct)$.

\begin{lemma}\label{mod equal alg} Both  $\cz$ and $\tilde\cz$ coincide.
\end{lemma}
\begin{proof}
 Obviously, $\cz\subseteq \tilde\cz$. We only need to check the reverse inclusion. 

 Recall that $Z_0(\ct)=Z(\ct)\cap \varphi(Z_0(\ggg))$ by the proof of Lemma \ref{Veldkamp mod}, and $\varphi(U(\ggg)^G)\subseteq Z(\ct)$ by definition. We have
\begin{align*}
\tilde\cz&=(Z_0(\ct),Z_1(\ct))=(\varphi(Z_0(\ggg))\cap Z(\ct), \varphi(U(\ggg)^G))\cr
&=(\varphi(Z_0(\ggg))\cap Z(\ct), \varphi(U(\ggg)^G)\cap Z(\ct))\cr
&\subseteq (\varphi(Z_0(\ggg)),\varphi(U(\ggg)^G))\cap Z(\ct)\cr
&=\varphi(Z_0(\ggg), U(\ggg)^G)\cap Z(\ct)\cr
&= \varphi(Z(\ggg))\cap Z(\ct).
\end{align*}
By Veldkamp's Theorem and  the fact $\varphi(Z_0(\ggg))=\varphi(Z_0(\tilde{\mathfrak{p}}))$ in \S\ref{pcen},  we have $\varphi(Z(\ggg))=\varphi(\ca)$. On the other hand, it follows from the definition of $\ca$ that $\cz=\varphi(\ca)\cap Z(\ct)$.
We finally have
$$\tilde\cz\subseteq \varphi(Z(\ggg))\cap Z(\ct)
=\varphi(\ca)\cap Z(\ct)=\cz.$$
The proof is completed.
\end{proof}


\section{Azumaya locus for $\ct(\ggg,e)$}\label{Azumaya locus}

In this section we consider the Azumaya property of the finite $W$-algebra $\ct(\ggg,e)$ in characteristic $p\gg0$.
\subsection{Background}\label{Azumaya background}
We first recall some facts on Azumaya algebras. For more details, the readers are referred to \cite[\S3]{BGl}.
\begin{defn}
A ring $A$  is called an Azumaya algebra over over  its center $Z(A)$ if $A$ is projective as a module  $A\otimes_{Z(A)} A^{\op}$; 
or equivalently,  $A$ is a finitely-generated projective $Z(A)$-module and the natural map $A\otimes_{Z(A)} A^{\op}\rightarrow \End_{Z(A)} A$ is an isomorphism.
\end{defn} In \cite{BGl}, Brown-Goodearl studied the irreducible modules of maximal dimension for a prime Noetherian algebra $A$ which is a finite module over its affine center $Z(A)$. When the base field is algebraically closed, the contractions of the annihilators of these modules to $Z(A)$ constitute the Azumaya locus  \begin{equation}\label{loc}
\mathcal{A}(A):=\{\mathbf{m}\in \Specm(Z(A))\mid A_{\mathbf{m}} \mbox{ is Azumaya over }Z(A)_{\mathbf{m}}\},
\end{equation}{\sl{where $\Specm(Z(A))$ stands for the spectrum of maximal ideals of $Z(A)$;}} $A_{\mathbf{m}}$ and $Z(A)_{\mathbf{m}}$ in \eqref{loc} denote the localization rings of $A$ and $Z(A)$ at $\mathbf{m}$, respectively.
    Moreover, $\mathcal{A}(A)$ is  a Zariski dense subset of  $\Specm(Z(A))$. Denote by $\mathcal{P}(A)$  the non-regular points in $\Specm(Z(A))$.

\begin{theorem} (\cite{BGl})\label{BGl 3.8}
Let $A$ be a prime Noetherian ring, module-finite over its center $Z(A)$. If $A$ is Auslander-regular and Macaulay, and $A_\ppp$ is Azumaya over $Z(A)_\ppp$  for all height $1$ prime ideals $\ppp$ of $Z(A)$ (that is, $\Specm(Z(A))\backslash \mathcal{A}(A)$ has codimension at least $2$ in
$\Specm(Z(A))$), then $\mathcal{A}(A)= \Specm(Z(A))\backslash \mathcal{P}(A)$.
\end{theorem}

By means of the above theory, Brown-Goodearl in \cite{BGl} revealed  the close connections between the smooth points in the spectrum of maximal ideals of the centers and the irreducible representations of maximal dimension for  the  quantized enveloping algebras, for the quantized function algebras at a root of unity, as well as for classical enveloping algebras in positive characteristic.

Recall that $\ct(\ggg,e)$ is a prime Noetherian ring, and module-finite over its center  $Z(\ct)$ by Lemma \ref{Noe Pri} and Theorem \ref{prem}(4). Thanks to \cite[Proposition 3.1]{BGl}, we know that for any irreducible $\ct(\ggg,e)$-module $V$, the dimension of $V$ is  maximal if and only if $\ct(\ggg,e)$ is Azumaya with respect to $\mathbf{m}=\text{Ann}_{\ct(\ggg,e)}(V)\cap Z(\ct)$ (where $\text{Ann}_{\ct(\ggg,e)}(V)$ denotes the annihilator of $V$ in $\ct(\ggg,e)$). That is, $\ct(\ggg,e)_{\mathbf{m}}$ is Azumaya over $Z(\ct)_{\mathbf{m}}$. We shall say that $\mathbf{m}$ lies in the Azumaya locus.

We shall study the following question.

\begin{question}\label{azumayaproperty} How about the relationship between the smooth locus and  the Azumaya locus reflecting the maximal dimensional representations of $\ct(\ggg,e)$?
 \end{question}


\subsection{Filtered rings}We will start with some basics on filtrations; see \cite[\S2]{Bj} for more details.

A ring $A$ with an (increasing) positive filtration $\{A_i\mid i\in \bbz_+\}$ such that $A_i\cdot A_j\subseteq A_{i+j}$ is called a positively filtered ring.
If $\bigcap_{i\in \bbz_+} A_i=0$ we say that $\{A_i\mid i\in \bbz_+\}$ is a separated filtration.
Then we can define a graded ring $\Gr(A)$ associated to $A$ as follows. Set $A_{-1}=0$, and define the additive group $\Gr(A):=\bigoplus_{i\in \bbz_+} A_i/A_{i-1}$. Given $\bar a_i\in A_i/A_{i-1}$ and $\bar a_j\in A_j/A_{j-1}$, the multiplication of $\bar a_i$ and $\bar a_j$ in $\Gr(A)$ is given by $\bar a_i \bar a_j:=\overline{a_ia_j}\in A_{i+j}/A_{i+j-1}$.

The filtration on a ring $A$ enables us to construct a filtered topology on $A$, using some distance function
$d(x,y)=2^i$ for a pair $x,y$ in $A$ such that $x-y$ 
belongs to
$A_i/A_{i-1}$. Since $\bigcap_{i\in \bbz_+}A_i=0$ we have $d(x,y)>0$ if $x\neq y$.
We may then refer to closed subsets of $A$ with respect to the filtered topology, and further define the corresponding induced topology and quotient topology as in \cite[\S2]{Bj}.

The filtration on $A$ satisfies the closure condition if
every finitely generated left or right ideal is a closed subset of $A$. The filtration satisfies the strong closure condition if the following holds: for every finite subset $x_1,\cdots,x_n$ of $A$ and
integers $i_1,\cdots,i_n$, it follows that the subsets of $A$ defined by
$A_{i_1}x_1+\cdots+A_{i_n}x_n$, respectively $x_1A_{i_1}+\cdots+x_nA_{i_n}$ are closed. The filtration on $A$ satisfies the weak comparison condition if the quotient topology is equal to the induced topology for every finitely-generated left or right ideal. The filtration on $A$ satisfies the comparison condition if the following holds: for every finite set $x_1,\cdots,x_n$ in $A$ there
exists an integer $m$ such that $A_j\cap(Ax_1+\cdots+Ax_n)\subseteq A_{j+m}x_1+\cdots+A_{j+m}x_n$
resp. $(x_1A+\cdots+x_nA)\cap A_j\subseteq x_1A_{j+m}+\cdots+x_n A_{j+m}$.

\begin{defn} A filtration $\{A_i\mid i\in \mathbb{Z}_+\}$ on $A$ is called Zariskian if $\Gr(A)$ is Noetherian and the following equivalent conditions hold.
  \begin{itemize}
\item[(1)] Strong closure condition;
\item[(2)] Closure and comparison conditions;
\item[(3)] Closure and weak comparison conditions.
\end{itemize}
\end{defn}

Furthermore, we can construct the graded Rees ring $\sR(A)$ associated to $A$. First define the additive group $\sR(A):=\bigoplus_{i\in \bbz_+} A_i$. Set $\sr_i$ to be the map from the filtered space $A_i$ onto the $i$-th homogeneous space $A_i$ of the associated graded space $\Gr_i(\sR(A))=A_i$.
The product in $\sR(A)$ is defined via
$$ \sr_i(v)\sr_j(w):=\sr_{i+j}(v\cdot w) \mbox{ for } v\in A_i, w\in A_j$$
for any pair $(i,j)\in \mathbb{Z}_+^2$.
Then the graded ring $\sR(A)$ is called the associated Rees ring of $A$.


Let $\sR(\ct)$ and $\sR(U)$ denote the Rees rings associated to $\ct(\ggg,e)$ and $U(\ggg,e)$, respectively. Then we have

\begin{lemma} \label{pbw filt} The following statements hold.
  \begin{itemize}
\item[(1)] The Rees rings $\sR(\ct)$ and $\sR(U)$ are Noetherian rings.
\item[(2)] The Kazhdan filtrations on $\ct(\ggg,e)$ and $U(\ggg,e)$ are both Zariskian.
 \end{itemize}
 \end{lemma}
\begin{proof}
%
In virtue of \eqref{grp} and Theorem \ref{prem}(4), \cite[Proposition 7.2]{AM} entails that $\text{gr}(\ct(\ggg,e))$ is a Noetherian ring. Taking Lemma \ref{Noe Pri}(2), Theorem \ref{prem}(2,3) and  \cite[Corollary 7.7]{AM} into consideration, we see that $U(\ggg,e)$ and $\text{gr}(U(\ggg,e))$ are also Noetherian.

Note that the Kazhdan filtrations on $\ct(\ggg,e)$ and $U(\ggg,e)$ are positive by \S\ref{212}, then it follows from \cite[Remark 2.22]{Bj} that the Kazhdan filtrations on  $\ct(\ggg,e)$ and $U(\ggg,e)$ are both Zariskian,  and the corresponding Rees rings $\sR(\ct)$ and $\sR(U)$ are Noetherian rings.
  \end{proof}

\subsection{Auslander-regular and Maculay rings}
Let $A$ be a ring, and $M$ an $A$-module. Then the grade number of $M$ is defined by
$$\sj_A(M):=\min\{ j\in \mathbb{Z}_+\mid \Ext^j_A(M, A)\ne 0\}.$$

\begin{defn}
Let $A$ be a Noetherian ring with finite global dimension. For every finitely-generated (right or left) $A$-module $M$, every $j\in\mathbb{Z}_+$ and every submodule $N$ of $\Ext^j_A(M,A)$, if one has $\sj_A(N)\geqslant\sj_A(M)$, i.e., $\Ext_A^i(N,A)=0$ for all non-negative integer $i<j$, then $A$ is said to be  Auslander-regular.
\end{defn}
\begin{defn}
A Macaulay ring $A$ is a ring for which the identity
 $$ \text{K.}\dim M +\sj_A(M)=\text{K.}\dim A$$
 holds for every non-zero finitely-generated $A$-module $M$, where $\text{K.}\dim M$ denotes the Krull dimension of $M$.
\end{defn}

Now we turn to the case of finite $W$-algebras. In fact, we have
\begin{prop} \label{Auslander Cond}
 The ring $\ct(\ggg,e)$ is Auslander-regular and Macaulay, with Krull and  global dimension at most  $\dim \ggg_e$.
\end{prop}
\begin{proof}
Let us first investigate the global dimension of $\ct(\ggg,e)$, which is denoted by $\text{gl.}\dim \ct(\ggg,e)$.
Recall that the Kazhdan filtration on $\ct(\ggg,e)$ is an increasing positive filtration. By \eqref{grse} we know that its graded algebra $\Gr(\ct(\ggg,e))$ is isomorphic to a polynomial algebra in $\dim\ggg_e$ variables, thus  $\Gr(\ct(\ggg,e))$ is a commutative Noetherian ring with global dimension $\dim\ggg_e$. According to \cite[Corollary 7.6.18]{MR}, we have $\text{gl.}\dim \ct(\ggg,e)\leqslant\dim\ggg_e$.

Recall that $\ct(\ggg,e)$ is a prime PI and Noetherian ring (see Lemma \ref{Noe Pri}(2,3) and Proposition \ref{PI and Central S}(4)), and module-finite over its center (see Theorem \ref{prem}(4)). Therefore, $\ct(\ggg,e)$ is a fully bounded Notherian ring (FBN) by \cite[Proposition 9.1(a)]{GlWa}. Note that the Kazhdan filtration on $\ct(\ggg,e)$ is Zariskian (see Lemma \ref{pbw filt}(2)). Since $\Gr(\ct(\ggg,e))$ is isomorphic to a polynomial algebra, then the ring $\Gr(\ct(\ggg,e))$ is Auslander-regular. By \cite[Theorem 3.9 and Remark 3.9]{Bj} we see that $\ct(\ggg,e)$ is also Auslander-regular. Furthermore, it is Macaulay by \cite[Corollary 4.5]{StZ}. Note that in this case, the global dimension of $\ct(\ggg,e)$ is equal to the Krull dimension (cf. \cite[Theorem 3.9]{T}). The proof is completed.
\end{proof}

\subsection{Azumaya locus of $\ct(\ggg,e)$}
Let $V$ be an irreducible $\ct(\ggg,e)$-module, and consider the corresponding central character (which is an algebra homomorphism) $\zeta_V : Z(\ct) \rightarrow \bbk$. 
As remarked in \S\ref{Azumaya background}, the  Azumaya locus in \eqref{loc} becomes
$$\mathcal{A}(\ct(\ggg,e))
 =\{\ker(\zeta_V)|\dim V=\mbox{the maximal dimension of all irreducible}~ \ct(\ggg,e)\mbox{-modules}\}.$$
It is well-known that Azumaya locus is a non-empty open subset.

To prove Theorem \ref{Azumaya Thm}, we  make some necessary preparations in the next subsections.

\subsection{The definition of Slodowy slices
}\label{Slo}
We first collect up some facts on Slodowy slices (see \cite{Jan3}, \cite{Pre1} and \cite{Slo}). 

 Recall that under the assumption $p\gg0$ the Jacobson-Marozov Theorem holds, which says, for every nilpotent element $e\in \ggg$ there exists an $\frak{sl}_2$-triple $(e,h,f)$ (cf. \cite[\S5.3-5.5]{Cart}). Furthermore, such an $\frak{sl}_2$-triple is unique up to the conjugation of  $C_G(e)^\circ$ (cf. {\sl{ibid}}).
Then we can set  $\cs=e+\kkk$ with $\kkk:=\ker(\ad\,f)$, which is usually called the Slodowy slice through the adjoint orbit of $e$ (see \cite[\S7.4]{Slo}).
Moreover, associated with the nilpotent $e$, one can have the weighted Dynkin diagram $\Delta(e)$ (cf. \cite[\S5.6]{Cart}),  as listed in \cite[Lemma 7.3.2]{Slo}. And two nilpotent elements lie in the same $G$-orbit if and only if they share the same Dynkin diagram (cf. \cite[\S5.6]{Cart}). The technique in the arguments of \cite[\S7.3-7.4]{Slo} is valid for our case (note that in \cite{Slo} there is a weaker requirement on $p$ which is zero or bigger than $4\textsf{h}-2$, where $\textsf{h}$ is the Coxeter number of $G$). In particular, \cite[Proposition 7.4.1 and Corollary 7.4.1]{Slo} are true for our case.

Recall that $r=\rank\,\ggg$, and we always assume that $p\gg0$. Since the Harish-Chandra's theorem shows that $\HC:~U(\ggg)^G\cong U(\hhh)^{W\centerdot}$ as $\bbk$-algebras, we continue to choose  $\{g_1,\cdots,g_r\}$ as a set of algebraically independent generators in $U(\ggg)^G$ as in \S\ref{3.2}. Let $\text{gr}_S(g_i)=\tilde g_i$ be the images of $g_i$ associated to the standard grading of $U(\ggg)$, and set $\kappa$ to be the Killing isomorphism. By \cite[\S7.12]{Jan3}, the Chevalley Restriction Theorem holds in our situation. In particular, $S(\ggg^*)^G\cong S(\hhh^*)^{W_\centerdot}$ as graded algebras under the standard grading. In virtue of \cite[\S3.12]{Slo}, the elements $\kappa(\tilde g_1),\cdots,\kappa(\tilde g_r)$ form a free generating set for $S(\ggg^*)^G$.

Under above assumption, we have the adjoint quotient map:
\begin{equation}\label{Steinberg map}
\varpi:\ggg\rightarrow \bba^r
\end{equation}
which sends $x\in\ggg$ to $(\kappa(\tilde g_1)(x),\cdots,\kappa(\tilde g_r)(x))$.
In virtue of \cite[\S7.3-7.4]{Slo}, we have the following result.
    \begin{lemma} (\cite{Slo})
    \label{Tran sli}  Under the assumption $p\gg0$, we consider the restriction of $\varpi$ to $\cs$ with
    $$\varpi_\cs: \cs\rightarrow  \bba^r.$$
Then the morphism $\varpi_\cs$ is faithfully flat.
\end{lemma}
\noindent
Consequently, $\varpi_\cs$ is surjective and all the  fibers of $\varpi_\cs$ have the same dimension $\sfd-r$ with $\sfd=\dim \ggg_e$.

Let $N=\#\Phi^+$ denote the number of positive roots of $G$ with respect to a maximal torus $T$.
   An element $X\in\ggg$ is called regular if the adjoint orbit $\text{Ad}\,G(X)$ has the maximal dimension which is equal to $2N$. Equivalently, the dimension of $\ggg_X$ (the centralizer of $X$ in $\ggg$) is the smallest one which is equal to $r$. By classical result, the regular elements constitute an open subset of $\ggg$, which is denoted by $\ggg_{\text{reg}}$. Moreover, it follows from \cite[Theorem 4.12]{Ve} that the complement of $\ggg_{\text{reg}}$ has codimension $3$ in $\ggg$.

   \subsection{Some results on Slodowy slices
   }  Now we continue the discussion on Slodowy slices under the assumption $p\gg0$. 
   We have the following result, which comes from a version over the field of complex numbers in \cite[Theorem 5.4]{Pre1}.

   \begin{theorem} (cf. \cite{Pre1} and \cite{Slo}) \label{premet} Suppose $G$ satisfies the assumption $p\gg0$. 
   Let $\cs$ be as in Lemma \ref{Tran sli} 
   and
   $$\xi=(\xi_1,\xi_2,\cdots,\xi_r)\in \bba^r, \qquad\varpi_\cs^{-1}(\xi)=\{X\in \cs\mid \varpi_\cs(X)=\xi\}.$$
   Then the following statements hold.
   \begin{itemize}
   \item[(1)] The closed set $\varpi_\cs^{-1}(\xi)$ of $\cs$ is irreducible and  of dimension $\sfd-r$.

   \item[(2)] The fiber $\varpi_\cs^{-1}(\xi)$ is normal. Let $Y\in \varpi_\cs^{-1}(\xi)$. Then $Y\in\ggg_\reg$ if and only if $Y$ is a smooth point of $\varpi_\cs^{-1}(\xi)$.
       \end{itemize}
       \end{theorem}

\begin{proof} (1) In \S\ref{Slo}, we have known that $\varpi_\cs$ is surjective and all the fibers have the same dimension $\mathsf{d}-r$. The irreducibility of $\varpi_\cs^{-1}(\xi)$ can be shown by the same arguments as in (3) and (4) of  the proof of \cite[Theorem 5.4]{Pre1}.

 (2) For the proof of the second statement, it can be accomplished by following Slodowy's arguments in \cite[\S5.2, Lemma and Remark 1]{Slo}. 
 \end{proof}

\begin{remark}\label{remnormal}

Set $\textsf{t}$ to be the transformation
\begin{equation*}
\begin{array}{llll}
\textsf{t}:&\kkk&\cong&\cs\\ &X&\mapsto&e+X,
\end{array}
\end{equation*}
which is an isomorphism of affine varieties. Then one obtains a morphism
\begin{align}\label{slicesmap}
\psi:=\varpi_\cs\circ \textsf{t}:\kkk\rightarrow\bba^r,\qquad X\mapsto(\psi_1(X),\cdots,\psi_r(X)).
\end{align}
The morphism $\psi$ is still faithfully flat with $\psi^{-1}(\xi)\cong\varpi_\cs^{-1}(\xi)$ for any $\xi\in\bba^r$, and the fiber $\psi^{-1}(\xi)$ is normal. One can describe $\cs$ via $\kkk$ (cf. \cite[\S5]{Pre1}).
\end{remark}

  Set $\cs_\reg:=\cs\cap \ggg_\reg$, which is clearly an open subset of $\cs$.

   \begin{prop} \label{Regular pts}  The complement of $\cs_\reg$ in $\cs$ has codimension at least $2$.
   \end{prop}
\begin{proof}
By Theorem \ref{premet} we know that $\varpi_\cs^{-1}(\xi)$ for any $\xi\in\bba^r$ is normal, and the smooth points of the fiber  coincide with the regular points of $\ggg$ falling into this fiber.
So the non-regular points in the fiber constitute a closed subset of codimension no less than $2$ (see \cite[Theorem II.5.1.3]{Sha}). Denote by $\cs_\sing$ the complement of $\cs_\reg$ in $\cs$. Then $\cs_\sing$ is a closed subset of  $\cs$. Set $l$ to be the dimension of $\cs_\sing$. We have a morphism $\varpi_{\cs_\sing}: \cs_\sing\rightarrow \bba^r$. Next, we take a maximal irreducible closed set $\cs^{\max}_\sing$ in $\cs_\sing$ with  $\dim\cs^{\max}_\sing=l$. Considering the restriction of $\varpi_{\cs_\sing}$ to $\cs^{\max}_\sing$, we have the morphism $\varpi_{\cs_\sing}^{\max}: \cs_\sing^{\max}\rightarrow \bba^r$. Set $\mathbf{A}$ to be the image of $\cs_\sing^{\max}$ under $\varpi_{\cs_\sing}^{\max}$. Then $\mathbf{A}$ is an irreducible variety of $\bba^r$, and let $a:=\dim \mathbf{A}\leqslant r$.

  According to \cite[Theorem I.6.3.7]{Sha}, there exists a non-empty open subset $U$ in $\mathbf{A}$ such that $\dim (\varpi_{\cs_\sing}^{\max})^{-1}(\xi)=l-a$ for any $\xi\in U\subseteq \mathbf{A} \subseteq\bba^r$. From the arguments above, we know that $(\varpi_{\cs_\sing}^{\max})^{-1}(\xi)$ ($\subseteq \varpi_\cs^{-1}(\xi)\cap \cs_\sing$)  has codimension $\geqslant 2$ in $\varpi_\cs^{-1}(\xi)$. By Theorem \ref{premet} again, we have $\dim\varpi_\cs^{-1}(\xi)=\sfd-r$. Hence $(\sfd-r)-(l-a)\geqslant 2$, i.e., $\sfd-l\geqslant 2+r-a\geqslant 2$. Therefore, $\cs_\sing$ has codimension not less than $2$ in $\cs$. The proof is completed.
\end{proof}

\subsection{Revisit of $Z_0(\ct)$}
Let us return to the $\mathds{k}$-algebra $Z_0(\ct)=Z_0(\tilde{\mathfrak{p}})^{\mathcal{M}}$, which is the $p$-center of the finite $W$-algebra $\ct(\ggg,e)$
. Recall that $Z_0(\ct)=Z(\ct)\cap \varphi(Z_0(\ggg))$ by \eqref{twoequation}.

As discussed in \S\ref{312} (also see \cite[\S8.2]{GT}),
 $Z_0(\ct)$ is defined as the algebra of regular functions on the Frobenius twist  $\bbk[(\chi+ \kappa(\kkk))^{(1)}]$.  Thus we can identify $\Specm(Z_0(\ct))$ with the spectrum of maximal ideals of the Frobenius twist $\bbk[(\kappa(\cs))^{(1)}]$.

Now we are finally in a position to prove Theorem \ref{Azumaya Thm}.

\subsection{The Proof of Theorem \ref{Azumaya Thm}}\label{azumaya proof}
\begin{proof}
 Recall that $\ct(\ggg,e)$ is a prime Noetherian ring by Lemma \ref{Noe Pri}(2,3), and module-finite over its center $Z(\ct)$ by Theorem \ref{prem}(4).
 According to  Theorem \ref{BGl 3.8} and Proposition \ref{Auslander Cond}, it suffices to verify that   $\mbox{codim}(\Specm ((Z(\ct))\backslash \mathcal{A}(\ct(\ggg,e))))\geqslant 2$.
We proceed to check this by several steps.

(a) The embedding homomorphism between $\bbk$-algebras: $Z_0(\ct)\hookrightarrow Z(\ct)$ gives rise to the finite dominant morphism
$$\Xi:\Specm(Z(\ct))\rightarrow \Specm(Z_0(\ct))\cong \Specm(\bbk[(\kappa(\cs))^{(1)}]).$$
We will identify $\Specm(Z_0(\ct))$ with $\Specm(\bbk[(\kappa(\cs))^{(1)}])$ in the following arguments.

(b) From the discussion in the proof of Proposition \ref{lessthan}, we know that all irreducible modules of maximal dimension for $\ct(\ggg,e)$ correspond to the ones of maximal dimension for $U_\eta(\ggg)$ with $\eta$ running through the set $\chi+\mmm^\perp$.

On the other hand, the coadjoint action map
$$\mathcal{M}\times\kappa(\cs)\rightarrow {\chi+\mmm^\perp}$$
gives rise to an isomorphism of affine varieties (cf. \cite[Lemma 3.2]{Pre3}, \cite[Lemma 2.1]{GG}, and \cite[Lemma 5.1]{GT}), where $\mathcal{M}$ is a unipotent subgroup of $G$ such that $\Lie(\mathcal{M})=\mmm$ as defined earlier. It follows that 
all regular elements of $\ggg^*$ in $\chi+\mmm^\perp$ come from the  regular ones of $\ggg^*$ in the saturation of $\kappa(\cs_\reg)$ under the $\mathcal{M}$-action.

Thanks to \cite[Proposition 3.15]{BGo}, any irreducible modules of $U_\eta(\ggg)$ associated with a regular $\eta\in \ggg^*$ are of the maximal dimension of irreducible $U(\ggg)$-modules.
Now we already identify $\Specm(Z_0(\ct))$ with $\Specm(\bbk[(\kappa(\cs))^{(1)}])$. Note that $(\kappa(\cs_\reg))^{(1)}=\kappa(\cs_\reg)$ as topological spaces.
Set $X=\{\mathcal{I}_s\in\Specm(\bbk[(\kappa(\cs))^{(1)}])\mid s\in\kappa(\cs_\reg)\}$, where $\mathcal{I}_s$ means the kernel of the $p$-central character $\lambda_s^0:Z_0(\ct)\rightarrow \bbk$ via $s\in\kappa(\cs_\reg)$ in the same sense as in Remark \ref{Central Character}
. Then $\Xi^{-1}(X)$ is contained in the Azumaya locus of $\ct(\ggg,e)$ by the very definition.

(c) By Proposition \ref{Regular pts}, the open subset $\kappa(\cs_\reg)$ of $\kappa(\cs)$ has complement of codimension at least $2$. Recall that $\Xi$ is a finite dominant morphism.  Hence
 $\Xi^{-1}(X)$ is an open subset of $\Specm(Z(\ct))$ with complement of codimension at least $2$. Therefore, we have the Azumaya locus of $\ct(\ggg,e)$ which has been aware of containing  $\Xi^{-1}(X)$,  consequently satisfies the condition that the codimension of its complement in $\Specm(Z(\ct))$ is at least $2$.
The proof is completed.
\end{proof}

\section{Centers of finite $W$-algebras revisited}\label{pf for thm 1.1}
Keep the notations as in the previous sections. In particular, let $Z(\ct)$ denote the center of $\ct(\ggg,e)$, and $\tilde\cz$  the subalgebra of $Z(\ct)$ generated by $Z_0(\ct)$ and $Z_1(\ct)$ as in \S\ref{On the centers of}.

In order to prove Theorem \ref{central thm}, we will take an analogue of  Veldkamp's strategy  in the reductive Lie algebra case (cf. \cite{Ve}).
We first show that both $Z(\ct)$ and $\tilde\cz$ have the same fraction field and  $\text{gr}(Z(\ct))$ is integral over $\text{gr}(\tilde{\cz})$, while $\text{gr}(\tilde{\cz})$ is integrally closed. Then we have $\text{gr}(\tilde{\cz})=\text{gr}(Z(\ct))$ and finally obtain that $Z(\ct)=\tilde\cz$.

\subsection{A key lemma}
Recall that in \S\ref{212} we have introduced the Kazhdan filtration on $U(\ggg)$ and its sub-quotients, and also their corresponding graded algebras.  
Write $\text{gr}(\tilde{\cz})$ and $\text{gr}(Z(\ct))$ for the graded algebras of $\tilde{\cz}$ and $Z(\ct)$ respectively.

Recall that $\{g_1,\cdots,g_r\}$ is a set of algebraically independent generators in $U(\ggg)^G$ with degree $m_1+1,\cdots,m_r+1$ associated to the standard grading of $U(\ggg)$, and $f_i=\varphi(g_i)$ with $1\leqslant i\leqslant r$ are elements in $Z_1(\ct)$ as in \S\ref{3.2}.
Then we have  the isomorphism $\bar\psi:\text{gr}(\ct(\ggg,e))\xrightarrow\sim S(\ggg_e)$ under the Kazhdan grading (see \eqref{grse}).

\begin{lemma}\label{complete int}
$\Specm(\text{gr}(\tilde\cz))$ is a (strict) complete intersection.
\end{lemma}
\begin{proof}  The arguments are almost the same as Veldkamp's strategy in \cite{Ve}. For the readers' convenience, we present the details and accomplish the arguments by steps.

(a) Let $x_1,\cdots,x_\sfd$ be a set of homogeneous basis of $\ggg_e$ as defined in \S\ref{211}.
Write $\bar f_1,\cdots,\bar f_r$ for the images of $f_1,\cdots,f_r$ in $\bar\psi(\text{gr}(\ct(\ggg,e)))$, then $\bar f_i$ is a polynomial in $x_1,\cdots,x_\sfd$ for $1\leqslant i\leqslant r$.
Due to \eqref{grp}, we have $\bar\psi:\text{gr}(Z_0(\ct))\xrightarrow\sim\bbk[x_1^p,\cdots,x_\sfd^p]$.   The morphism $\bar\psi$ finally gives rise to the following equations
\begin{equation}\label{grisom}
\bar\psi(\text{gr}(\tilde\cz))=\bar\psi(\text{gr}(Z_0(\ct)\cdot Z_1(\ct)))=\bbk[x_i^p,\bar{f}_j\mid1\leqslant i\leqslant\sfd,1\leqslant j\leqslant r].
\end{equation}
So there exists a surjective homomorphism
\begin{equation}\label{psi}
\omega: \bbk[Y_1,\cdots,Y_\sfd,Z_1,\cdots,Z_r]/(Z_1^p-F_1,\cdots,Z_r^p-F_r)\rightarrow\bar\psi(\text{gr}(\tilde\cz))
\end{equation}
with
\begin{equation}\label{psipro}
\begin{array}{ll}
\omega(Y_i)=x_i^p&(1\leqslant i\leqslant\sfd),\\
\omega(Z_i)=\bar{f}_i&(1\leqslant i\leqslant r).
\end{array}
\end{equation}
Here $Y_1,\cdots,Y_\sfd,Z_1,\cdots,Z_r$ denote algebraically independent variables over $\bbk$,
and $F_i\in\bbk[Y_1,\cdots,Y_\sfd]$ satisfies
 \begin{equation}\label{F_i}
F_i(x_1,\cdots,x_\sfd)=\bar f_i.
\end{equation}

(b) We will further show that the map $\omega$ defined in \eqref{psi} is  injective.

(Step 1)
For $\mathbf{t}:=(t_1,t_2,\cdots,t_r)\in\Lambda_r$,
let $b_{\mathbf{t}}(\Phi(x_1),\cdots,\Phi(x_\sfd))$ be a monomial in $\bbk[\Phi(x_1),\cdots,\Phi(x_\sfd)]$ (where $\Phi(x_i)$ is defined as in \eqref{xkzp}), and set $$a_\mathbf{t}(x_1^p,\cdots,x_\sfd^p):=\bar\psi(\text{gr}(b_{\mathbf{t}}(\Phi(x_1),\cdots,\Phi(x_\sfd)))),$$
where $\text{gr}(b_{\mathbf{t}}(\Phi(x_1),\cdots,\Phi(x_\sfd)))$ means the gradation of $b_{\mathbf{t}}(\Phi(x_1),\cdots,\Phi(x_\sfd))$ under the Kazhdan grading.
Suppose that we are given an equation
\begin{equation}\label{ateq}
\sum_{\mathbf{t}\in\Lambda_r}a_\mathbf{t}(x_1^p,\cdots,x_\sfd^p)\bar f_1^{t_1}\cdots\bar f_r^{t_r}=0
\end{equation}
in $\bar\psi(\text{gr}(\tilde\cz))$. We  show below that $a_\mathbf{t}(x_1^p,\cdots,x_\sfd^p)=0$ for all $\mathbf{t}\in\Lambda_r$.

(Step 2)
Since 
$x_1,\cdots,x_\sfd$ is a basis of $\ggg_e$, 
let $x^*_1,\cdots,x^*_\sfd$ denote the linear functions on $\kkk$  defined via $x^*_i(y)=(x_i,y)$ for $y\in\kkk$. 
 Recall that in (\ref{slicesmap}) we have
\begin{align*}
\psi:=\varpi_\cs\circ \textsf{t}:\kkk\rightarrow\bba^r,\qquad X\mapsto(\psi_1(X),\cdots,\psi_r(X)).
\end{align*}
Those $\psi_i$ naturally become functions in variables $x_i^*$, $i=1,\ldots,\sfd$.
By the same arguments as in \cite[Propositon 5.2]{Pre1},  we know that $(\mathsf{d}\psi)_X$ is surjective  for $X\in \kkk$ with $e+X\in \ggg_{\text{reg}}$ (In \cite{Pre1} the case is considered over $\bbc$. By careful inspection, one can easily conclude that it also goes through over the field in characteristic $p\gg0$).
Hence the Jacobian  matrix
\begin{equation*}\left(\partial \psi_i\over
\partial x^*_j\right)_{\tiny\begin{array}{l}
1\leqslant i\leqslant r\\1\leqslant j\leqslant \sfd
\end{array}}
\end{equation*}
 has rank $r$ in $\kkk$.

 In aid of the non-degenerate bilinear form $(\cdot,\cdot)$, we identify the polynomial function ring $\bbk[\kkk]$ on $\kkk$ with the symmetric algebra  $S(\ggg_e)$ on $\ggg_e$, correspondingly identify $\psi_i$ with $\bar f_i$ (Here $\psi_i$ and $\bar f_i$ are homogeneous elements under the Kazhdan grading. Note that up to the transformation $\textsf{t}$,
$\psi_i$ is obtained by 
restriction of $\kappa(\text{gr}_S(g_i))=\kappa(\tilde g_i)$ on $\cs$ under the standard grading with homogeneous degree $m_i+1$ (see \S\ref{Slo}), and then the same $\psi_i$ is still a homogeneous element with degree $2m_i+2$ under the Kazhdan grading by \cite[\S5.1]{Pre1}
). And it is finally known that
 \begin{equation*}\left(\partial \bar f_i\over
\partial x_j\right)_{\tiny\begin{array}{l}
1\leqslant i\leqslant r\\1\leqslant j\leqslant \sfd
\end{array}}
\end{equation*}
has rank $r$. It follows that the elements $\bar f_1,\cdots,\bar f_r$ are $p$-independent in the quotient field of $\bbk[x_1,\cdots,x_\sfd]$ (see \cite[Page 180]{Jacob}). Therefore,
$a_\mathbf{t}(x_1^p,\cdots,x_\sfd^p)=0$ for
all $\mathbf{t}=(t_1,t_2,\cdots,t_r)\in\Lambda_r$ in \eqref{ateq}.

(Step 3) Let us return to the map $\omega$ in \eqref{psi}. From the definition of $\omega$ we see that the $\bbk$-algebra $\bbk[Y_1,\cdots,Y_\sfd,Z_1,\cdots,Z_r]/(Z_1^p-F_1,\cdots,Z_r^p-F_r)$ is spanned by $Z_1^{t_1}\cdots Z_r^{t_r}$ with $\mathbf{t}\in\Lambda_r$ over $\bbk[Y_1,\cdots,Y_\sfd]$.
Note that $\omega$ satisfies (\ref{psipro}) and (\ref{F_i}), and all $x_i$ for $i=1,\cdots,\sfd$ are algebraically independent.
Hence, the injectivity of $\omega$ follows from the discussion in (Step 1) and (Step 2).

(c) Taking \eqref{grisom} into consideration, we finally get
\begin{equation}\label{gre}
\text{gr}(\tilde\cz)\cong\bbk[Y_1,\cdots,Y_\sfd,Z_1,\cdots,Z_r]/(Z_1^p-F_1,\cdots,Z_r^p-F_r)
\end{equation}
as $\bbk$-algebras.
Hence $\Specm(\text{gr}(\tilde\cz))$ is a (strict) complete intersection.
\end{proof}

Now we are in a position to prove Theorem \ref{central thm}.

\subsection{The proof of Theorem \ref{central thm}(1)}\label{pro}
Still set $n=\dim\ggg$ and $r=\rank\,\ggg$.  It is sufficient to prove that $Z(\ct)$ coincides with $\tilde{\cz}$.  The arguments are divided into a couple of steps.

  (Step 1) Recall the fractional fields  $\bbf_0=\Frac(Z_0(\ct))$ and $\bbf=\Frac(Z(\ct))$. Obviously $\bbf_0$ is  a subfield of $\bbf$.
  Set $\cq$ to be  the fractional field of $\tilde\cz$. By Lemmas \ref{Har Chand Cent} and \ref{mod equal alg}, $\cq/\bbf_0$ is a finite field extension with extension degree $[\cq:\bbf_0]=p^{r}$. We  first prove that $\bbf$ coincides with its subfield $\cq$.

 Let us investigate the field extension  $\bbf\slash\bbf_0$. Consider the fractional ring of $\ct(\ggg,e)$:
 $$Q(\ct)=\ct(\ggg,e)\otimes_{Z(\ct)}\bbf.$$
 Denote by $[Q(\ct):\bbf]$ the dimension of $Q(\ct)$ over $\bbf$.  From Proposition \ref{centersimple}(4) we see that $[Q(\ct):\bbf]\geqslant p^\ell$ with $\ell=n-r-\dim  G.e$. On the other hand, by Premet's result in \cite[Remark 2.1]{Pre2} we know that $\ct(\ggg,e)$ is a free $Z_0(\ct)$-module~ of rank $p^{n-\dim  G.e}$. Hence $Q(\ct)$ over $\bbf_0$ has dimension  $[Q(\ct):\bbf_0]=p^{n-\dim  G.e}$. Consequently, we have
 \begin{align*}
 [\bbf:\bbf_0]&=[Q(\ct):\bbf_0]\slash [Q(\ct):\bbf]\cr
 &\leqslant p^{n-\dim  G.e}\slash p^\ell\cr
 &=p^{r}.
 \end{align*}
 Combining this with $[\cq:\bbf_0]=p^r$, we have  $\cq=\bbf$.

(Step 2)
As $\tilde{\cz}\subseteq Z(\ct)$ and $\bbk$ is contained in $\tilde{\cz}$, by inductive arguments on the degrees we can deduce that the claim that
$$\tilde{\cz}=Z(\ct)$$ amounts to  the claim that
\begin{equation}\label{griso}
\text{gr}(\tilde{\cz})=\text{gr}(Z(\ct))
\end{equation}
under the Kazhdan grading. 

In (Step 1) we have seen that $Z(\ct)$ and $\tilde{\cz}$ have the same fraction field. From this one easily derives that their graded algebras $\text{gr}(Z(\ct))$ and $\text{gr}(\tilde{\cz})$ also share the same fraction field. It follows from Proposition \ref{PI and Central S}(3) that $\text{gr}(Z(\ct))$ is integral over $\text{gr}(Z_0(\ct))$, so, {\it{a fortiori}}, $\text{gr}(Z(\ct))$ is integral over $\text{gr}(\tilde{\cz})$. Therefore, to prove \eqref{griso} it is sufficient to show that $\text{gr}(\tilde{\cz})$ is integrally closed.  

(Step 3)  Keep in mind  the algebra isomorphism in (\ref{gre}). Consider a finite morphism $\iota$ from $\Specm(\text{gr}(\tilde\cz))$ onto affine space $\bba^\sfd$ arising from the injection
\begin{equation}\label{inject}
\iota: \bbk[Y_1,\cdots,Y_\sfd]\rightarrow
\bbk[Y_1,\cdots,Y_\sfd,Z_1,\cdots,Z_r]/(Z_1^p-F_1,\cdots,Z_r^p-F_r).
\end{equation}
The singular points of $\Specm(\text{gr}(\tilde\cz))$ are the points where the functional matrix
$$\bigg({\partial F_i\over\partial Y_j} \bigg)_{\tiny\begin{array}{l}
1\leqslant i\leqslant r\\1\leqslant j\leqslant \sfd
\end{array}}$$
has rank less than $r$. By definition, this functional matrix is the same as
\begin{equation}\label{sec function matrix}
\bigg({\partial \kappa(\tilde g_i)|_\cs\over\partial \upsilon_j} \bigg)_{\tiny\begin{array}{l}
1\leqslant i\leqslant r\\1\leqslant j\leqslant \sfd
\end{array}}
\end{equation}
defined by $\varpi_\cs$ (see Lemma \ref{Tran sli}
), where $\kappa(\tilde g_1)|_\cs,\cdots,\kappa(\tilde g_r)|_\cs$ are the functions $\kappa(\tilde g_1),\cdots,\kappa(\tilde g_r)$ restricted to the ones on $\cs$, and    $\upsilon_1,\cdots,\upsilon_\sfd$ denote linear coordinates on $\cs$. The points with rank less than $r$ defined in (\ref{sec function matrix})  are just non-regular points in $\cs\subseteq\ggg$. By Proposition \ref{Regular pts}, the complement of $\cs_\reg$ in $\cs$ has codimension not less than $2$.
Since $\iota$ is a finite morphism, the set of singular points of $\Specm(\text{gr}(\tilde\cz))$  has codimension not less than $2$. This is to say, $\Specm(\text{gr}(\tilde\cz))$ is smooth outside of a subvariety of codimension $2$.

(Step 4) By Lemma \ref{complete int}, $\Specm(\text{gr}(\tilde\cz))$ is a (strict) complete intersection. According to Serre's theorem \cite[Chapter 4]{Se}, a complete intersection variety is normal if it is smooth outside of a subvariety of codimension 2.
By (Step 3), $\Specm(\text{gr}(\tilde\cz))$ is normal, i.e., $\text{gr}(\tilde{\cz})$ is integrally closed, which implies  (\ref{griso}).
According to the analysis in  (Step 2),  we have known that $\tilde\cz=Z(\ct)$.

\subsection{The proof of Theorem \ref{central thm}(2)}\label{pro2}
 To prove Statement (2), we mainly take use of the results in \cite{BGo}, and also the ones we obtained in the previous sections.

Keep the notations as before. We claim that $Z_0(\ct)\cap Z_1(\ct)$ is a polynomial
algebra of rank $r$, and $Z_1(\ct)$ is a free $Z_0(\ct)\cap Z_1(\ct)$-module with basis $\{f_1^{t_1}\cdots f_r^{t_r}\mid0\leqslant t_i\leqslant p-1, 1\leqslant i\leqslant r\}$.

By \cite[Theorem 3.5(2)]{BGo},
 $Z_0(\ggg)\cap Z_1(\ggg)$ is a polynomial algebra of rank $r$, and $Z_1(\ggg)$ is a free $Z_0(\ggg)\cap Z_1(\ggg)$-module of rank $p^r$ with a basis consisting of all $g^{\mathbf{t}}=g_1^{t_1}\cdots g_r^{t_r}$ with $0\leqslant t_i\leqslant p-1$ for all $i$ (the condition on $p$ is automatically satisfied in the present case).
In virtue of Lemma \ref{Har Chand Cent}, the restriction of map $\varphi$  to $Z_1(\ggg)$ is an isomorphism. As $Z_0(\ggg)\cap Z_1(\ggg)$ is a subalgebra of $Z_1(\ggg)$, it is  an immediate consequence that $\varphi(Z_0(\ggg)\cap Z_1(\ggg))\cong Z_0(\ggg)\cap Z_1(\ggg)$ as $\bbk$-algebras. Therefore, $\varphi(Z_0(\ggg)\cap Z_1(\ggg))$ is also a polynomial algebra of rank $r$, and  $\varphi(Z_1(\ggg))=Z_1(\ct)$ is a free $\varphi(Z_0(\ggg)\cap Z_1(\ggg))$-module of rank $p^r$ with a basis consisting of all $f^{\mathbf{t}}=\varphi(g^{\mathbf{t}})=f_1^{t_1}\cdots f_r^{t_r}$ for $\mathbf{t}\in\Lambda_r$.

Now we will show that \begin{equation}\label{varz01}
\varphi(Z_0(\ggg)\cap Z_1(\ggg))=Z_0(\ct)\cap Z_1(\ct).
\end{equation}
First we have $\varphi(Z_0(\ggg)\cap Z_1(\ggg))\subseteq\varphi(Z_0(\ggg))\cap\varphi(Z_1(\ggg))$ by definition.  On the other hand, for any $x\in \varphi(Z_0(\ggg))\cap\varphi(Z_1(\ggg))$, we can write
$x=\sum_{\mathbf{t}\in\Lambda_r}a_{\mathbf{t}}f_1^{t_1}\cdots f_r^{t_r}$ for $a_{\mathbf{t}}\in\varphi(Z_0(\ggg)\cap Z_1(\ggg))$ because $\varphi(Z_1(\ggg))$ is already proved to be a free $\varphi(Z_0(\ggg)\cap Z_1(\ggg))$-module of rank $p^r$ with a basis consisting of all $f^{\mathbf{t}}=f_1^{t_1}\cdots f_r^{t_r}$ for $\mathbf{t}\in\Lambda_r$. So one can obtain that \begin{equation}\label{label}
(a_{(0,0,\cdots,0)}-x)+\sum_{\mathbf{t}\in\Lambda_r,\mathbf{t}\neq(0,0,\cdots,0)}a_{\mathbf{t}}f_1^{t_1}\cdots f_r^{t_r}=0.
\end{equation}Note that $\varphi(Z_0(\ggg)\cap Z_1(\ggg))\subseteq \varphi(Z_0(\ggg))\cap\varphi(Z_1(\ggg))\subseteq\varphi(Z_0(\ggg))\cap\ct(\mathfrak{g},e)=Z_0(\ct)$ by \eqref{twoequation} and Lemma \ref{Har Chand Cent}, and  $x\in Z_0(\ct)$ by definition. Then
it follows from Statement (1) that all the coefficients in \eqref{label} must be zero. In particular, $x=a_{(0,0,\cdots,0)}\in\varphi(Z_0(\ggg)\cap Z_1(\ggg))$, thus we have $\varphi(Z_0(\ggg))\cap\varphi(Z_1(\ggg))\subseteq\varphi(Z_0(\ggg)\cap Z_1(\ggg))$. Therefore, we obtain that $\varphi(Z_0(\ggg)\cap Z_1(\ggg))=\varphi(Z_0(\ggg))\cap\varphi(Z_1(\ggg))$. Moreover, taking \eqref{twoequation}, Lemma \ref{Har Chand Cent} and the discussion in \S\ref{pcen} into consideration, we further have
that \begin{equation}\label{red}
\varphi(Z_0(\ggg))\cap\varphi(Z_1(\ggg))=Z_0(\tilde{\mathfrak{p}})\cap\ct(\mathfrak{g},e)\cap Z_1(\ct)=Z_0(\ct)\cap Z_1(\ct),
\end{equation}then \eqref{varz01} follows.

Actually we can further obtain that $Z_1(\ct)$ is a complete intersection over $Z_0(\ct)\cap Z_1(\ct)$.
This is because $Z_1(\ggg)$ is a complete intersection over $Z_0(\ggg)\cap Z_1(\ggg)$ by \cite[Theorem 3.5(3)]{BGo}, and
Lemma \ref{Har Chand Cent} entails that the map $\varphi$ restricted to $Z_1(\ggg)$ is an isomorphism, then $Z_1(\ggg)\cong Z_1(\ct)$ and $Z_0(\ggg)\cap Z_1(\ggg)\cong Z_0(\ct)\cap Z_1(\ct)$ as $\bbk$-algebras.

By summing up the above along with the first statement, the proof of the second statement is completed.

\begin{rem}
As an immediate consequence of the proof of Theorem \ref{central thm}(1), we have  $2m=\ell$ in Proposition \ref{centersimple}(4).
\end{rem}
\begin{rem}\label{similarresu}
All the discussion in this section goes through for the extended finite $W$-algebra $U(\ggg,e)$, i.e.,
let $Z(U(\ggg,e))$ be the center of the $\mathds{k}$-algebra $U(\ggg,e)$. Then we have the following results.
\begin{itemize}
\item[(1)] The $\mathds{k}$-algebra $Z(U(\ggg,e))$ is generated by  $Z_0(\tilde{\mathfrak{p}})$ and $Z_1(\ct)$. More precisely, $Z(U(\ggg,e))$ is a free $Z_0(\tilde{\mathfrak{p}})$-module of rank $p^r$ with a basis $f_1^{t_1}\cdots f_r^{t_r}$, where $(t_1,\cdots,t_r)$ runs through $\Lambda_r$.
\item[(2)] The multiplication map $\nu:~Z_0(\tilde{\mathfrak{p}})\otimes_{Z_0(\tilde{\mathfrak{p}})\cap Z_1(\ct)} Z_1(\ct)\rightarrow Z(U(\ggg,e))$ is an isomorphism of
algebras.
\end{itemize}
\end{rem}

\subsection*{Acknowledgment} The first-named author is grateful to Bingyong Xie for his helpful discussion on the proof of Proposition \ref{Regular pts}. We are grateful to Weiqiang Wang for his helpful comments and his aid to improve the exposition of an old version of the manuscript. At last, we express our deep thanks to Sasha Premet for reading the  manuscript carefully and his valuable comments.

\end{document}